\documentclass{amsart}%
\usepackage{amsmath}
\usepackage{amsfonts}
\usepackage{amssymb}
\usepackage{graphicx}%
\newtheorem{lemma}{Lemma}[section]
\newtheorem{theorem}[lemma]{Theorem}
\newtheorem{remark}[lemma]{Remark}

\newtheorem{coro}[lemma]{Corollary}
\newtheorem{definition}[lemma]{Definition}
\newtheorem{example}[lemma]{Example}

\title[Massera's theorem for \ldots Scalar Differential Equations]{Massera's theorem for Asymptotically Periodic Scalar Differential Equations}

\author{David Cheban}
\address[D. Cheban]{State University of Moldova\\
Faculty of Mathematics and Informatics\\
Laboratory of Fundamental and Applied Mathematics\\
A. Mateevich Street 60\\
MD--2009 Chi\c{s}in\u{a}u, Moldova} \email[D.
Cheban]{david.ceban@usm.md, davidcheban@yahoo.com}

\date{\today}
\subjclass{34C12, 34C25, 34C27, 34D05, 37B55, 37C65, 39A23}
\keywords{Asymptotically periodic solution; scalar differential
equations, monotone dynamical systems; cocycles}

\begin{document}

\begin{abstract}

The aim of this paper is studying the problem of existence of
asymptotically periodic solutions of the scalar differential
equation $x'=f(t,x),$ where $f:\mathbb R\times \mathbb R\to
\mathbb R$ is a continuous asymptotically $\tau$-periodic
function. We prove that every bounded on semi-axis $\mathbb R_{+}$
solution $\varphi$ of this equation is S-asymptotically
$\tau$-periodic, i.e., $\lim\limits_{t\to
+\infty}|\varphi(t+\tau)-\varphi(t)|=0$. This statement is a
generalization of the well-known Massera's theorem for
asymptotically periodic scalar differential equations. We also
establish a similar statement for scalar difference equations.
\end{abstract}

\maketitle

\section{Introduction}\label{S1}

Denote by $\mathbb R =(-\infty,+\infty)$, $\mathbb
R_{+}=[0,+\infty)$ and $C(\mathbb R\times \mathbb R,\mathbb R)$
the space of all continuous functions $f:\mathbb R\times \mathbb
R\to \mathbb R$ equipped with the compact open topology. Denote by
$f^{h}$ ($h\in \mathbb R$) the $h$-translation of $f$ with respect
to time, i.e., $f^{h}(t,x):=f(t+h,x)$ for any $(t,x)\in \mathbb
R\times \mathbb R$. Consider a scalar differential equation
\begin{equation}\label{eqI1}
x'=f(t,x),
\end{equation}
where $f\in C(\mathbb R\times \mathbb R,\mathbb R)$. Along with
equation (\ref{eqI1}) we will consider the family of equations
\begin{equation}\label{eqI2}
y'=g(t,y),
\end{equation}
where $g\in H^{+}(f):=\overline{\{f^{h}|\ h\in \mathbb R_{+}\}}$
and by bar the closure in the space $C(\mathbb R\times \mathbb
R,\mathbb R)$ is denoted.

The function $f$ (respectively, equation (\ref{eqI1}) is called
positively regular, if for any $g\in H^{+}(f)$ the equation
(\ref{eqI2}) admits a unique solution $\varphi(t,v,g)$ passing
through the point $v\in \mathbb R$ at the initial moment $t=0$ and
defined on $\mathbb R_{+}$.

It is well known the following Massera's result \cite{Mas_1950}
(see also \cite[Ch.XII]{Fin_1974} and \cite[Ch.II]{Plis_64}).

\begin{theorem}\label{thI1} Assume that the following conditions
are fulfilled:
\begin{enumerate}
\item the function $f\in C(\mathbb R\times \mathbb R,\mathbb R)$
is $\tau$-periodic ($\tau >0$) in time, i.e., $f(t+\tau,x)=f(t,x)$
for any $(t,x)\in \mathbb R\times \mathbb R$; \item $f$ is
positively regular; \item the equation (\ref{eqI1}) admits a
bounded on $\mathbb R_{+}$ solution $\varphi(t,u_0,f)$.
\end{enumerate}

Then the solution $\varphi(t,u_0,f)$ is asymptotically
$\tau$-periodic, i.e., there exists a $\tau$-periodic function
$p:\mathbb R\to \mathbb R$ such that
\begin{equation}\label{eqI3}
\lim\limits_{t\to +\infty}|\varphi(t,u_0,f)-p(t)|=0 .\nonumber
\end{equation}
\end{theorem}

In this paper we study the following problem.

\textbf{Problem.} Assume that the following conditions hold:
\begin{enumerate}
\item the function $f\in C(\mathbb R\times \mathbb R,\mathbb R)$
is asymptotically $\tau$-periodic in time, i.e., there exist
functions $P, R\in C(\mathbb R\times \mathbb R,\mathbb R)$ such
that
\begin{enumerate}
\item $f(t,x)=P(t,x)+R(t,x)$ for any $(t,x)\in \mathbb R\times
\mathbb R$; \item $P(t+\tau,x)=P(t,x)$ for any $(t,x)\in \mathbb
R\times \mathbb R$; \item
\begin{equation}\label{eqI4}
\lim\limits_{t\to +\infty}|R(t,x)|=0\nonumber
\end{equation}
uniformly with respect to $x$ on every compact subset $K$ from
$\mathbb R$.
\end{enumerate}
\item $f$ is positively regular; \item the equation (\ref{eqI1})
admits a bounded on $\mathbb R_{+}$ solution $\varphi(t,u_0,f)$.
\end{enumerate}

\emph{Question.} Will there be the solution $\varphi(t,u_0,f)$
asymptotically $\tau$-periodic?

In general, the answer to this question is negative (see Example
\ref{exP1}). The main result of this paper in the following
theorem is contained.

\begin{theorem}\label{thI2}
Suppose that the following conditions hold:
\begin{enumerate}
\item the function $f\in C(\mathbb R\times \mathbb R,\mathbb R)$
is asymptotically $\tau$-periodic ($\tau >0$) in time; \item $f$
is positively regular; \item the equation (\ref{eqI1}) admits a
bounded on $\mathbb R_{+}$ solution $\varphi(t,u_0,f)$.
\end{enumerate}

Then the solution $\varphi(t,u_0,f)$ is $S$-asymptotically
$\tau$-periodic \cite{HPT_2008}, i.e.,
\begin{equation}\label{eqI03}
\lim\limits_{t\to
+\infty}|\varphi(t+\tau,u_0,f)-\varphi(t,u_0,f)|=0 .\nonumber
\end{equation}
\end{theorem}

The paper is organized as follows.

In the second section we collect some known notions and facts from
theory of dynamical systems (both autonomous and nonautonomous).

The third section is dedicated to the study the $S$-asymptotically
$\tau$-periodic motions of one-dimensional monotone nonautonomous
dynamical systems.

In the fourth section we construct an example which show that
Massera's theorem for asymptotically periodic scalar differential
equations in general is false.

The fifth section is dedicated to the study the problem of
existence of asymptotically $\tau$-periodic motions for
one-dimensional monotone nonautonomous dynamical systems.

In the sixth section we establish some properties of
asymptotically periodic functions in the framework of shift
dynamical systems on the space of continuous functions (Bebutov's
dynamical systems).

The seventh section is dedicated to the applications of our
general results obtained in Sections \ref{S3}-\ref{S6} for scalar
differential and difference equations.

\section{Preliminaries}\label{S2}

Let $X$ and $Y$ be two complete metric spaces, let $\mathbb Z
:=\{0,\pm 1, \pm 2, \ldots \}$, $\mathbb S =\mathbb R$ or $\mathbb
Z$,  $\mathbb S_{+}=\{t \in \mathbb S |\quad t \ge 0 \}$ and
$\mathbb S_{-}=\{t \in \mathbb S| \quad t \le 0 \}$. Let $\mathbb
T\subseteq \mathbb S$ be a sub-semigroup of $\mathbb S$ with
$\mathbb S_{+}\subseteq \mathbb T$, $\mathbb T_{+}:=\{t\in \mathbb
T|\ t\ge 0\}$, $(X,\mathbb S_{+},\pi)$ (respectively, $(Y,\mathbb
S, \sigma )$) be an autonomous one-sided (respectively, two-sided)
dynamical system on $X$ (respectively, on $Y$).

Let $(X,\mathbb T,\pi)$ be a dynamical system and
$\pi(t,x)=\pi^{t}x=xt$.

\begin{definition}\label{defSP1} A point $x\in X$ (respectively, a motion $\pi(t,x)$) is
said to be:
\begin{enumerate}
\item[-] stationary, if $\pi(t,x)=x$ for any $t\in \mathbb S_{+}$;
\item[-] $\tau$-periodic ($\tau >0$ and $\tau \in \mathbb S_{+}$),
if $\pi(\tau,x)=x$; \item[-] asymptotically stationary
(respectively, asymptotically $\tau$-periodic), if there exists a
stationary (respectively, $\tau$-periodic) point $p\in X$ such
that
\begin{equation}\label{eqAP1*}
\lim\limits_{t\to \infty}\rho(\pi(t,x),\pi(t,p))=0.\nonumber
\end{equation}
\end{enumerate}
\end{definition}

\begin{theorem}\label{thAAP1}\cite[Ch.I]{Che_2009} A point $x\in X$ is asymptotically $\tau$-periodic if and
only if the sequences $\{\pi(k\tau,x)\}_{k=0}^{\infty}$ converges.
\end{theorem}

\begin{definition}\label{defLS1} A point $\tilde{x}\in X$ is said
to be $\omega$-limit for $x\in X$ if there exists a sequence
$\{t_k\}\subset \mathbb S_{+}$ such that $t_k\to +\infty$ and
$\pi(t_k,x)\to \tilde{x}$ as $k\to \infty$.
\end{definition}

Denote by $\omega_{x}$ the set of all $\omega$-limit points of
$x\in X$.

\begin{definition}\label{defLS2} A point $x$ is called positively Lagrange
stable, if the semi-trajectory $\Sigma_{x}^{+}:=\{\pi(t,x)|\ t\in
\mathbb S_{+}\}$ is a precompact subset of $X$.
\end{definition}

\begin{theorem}\label{th1.3.9}\cite[Ch.I]{Che_2020} Let $x\in X$ be positively Lagrange stable and $\tau\in\mathbb T\ (\tau >0)$. Then the following
statements are equivalent:
\begin{enumerate}
\item[a.]
\begin{equation}\label{eqS*} \lim\limits_{t\to
+\infty}\rho(\pi(t+\tau,x),\pi(t,x))=0;
\end{equation}
\item[b.] any point $p\in\omega_{x}$ is $\tau$-periodic.
\end{enumerate}
\end{theorem}
\begin{proof}
Let $p$ be an arbitrary point from $\omega_{x}$. Then there exists
a sequence $\{t_k\}\subset \mathbb T$ such that $t_k\to +\infty$
and
\begin{equation}\label{eqAP1}
\lim\limits_{k\to \infty}\rho(\pi(t_k,x), p)=0.
\end{equation}
Note that
\begin{equation}\label{eqAP2}
\rho(\pi(\tau,p),p)\le
\rho(\pi(\tau,p),\pi(\tau,\pi(t_k,p)))+\rho(\pi(t_k+\tau,x),\pi(t_k,x)))+\rho(\pi(t_k,x),p)
\end{equation}
for any $k\in \mathbb N$. Passing to the limit in (\ref{eqAP2}) as
$k\to \infty$ and taking into account the relations (\ref{eqS*})
and (\ref{eqAP1}) we obtain $\pi(\tau,p)=p$, i.e., every point
from $\omega_{x}$ is $\tau$-periodic.

Now, we will show the implication $b.\to a.$. If we suppose that
it is not true, then there are $\varepsilon_0>0$ and $t_n\to
+\infty$ as $n\to \infty$ such that
\begin{equation}\label{eqS.9.1}
\rho(\pi(t_n+\tau,x),\pi(t_n,x))\ge \varepsilon_0
\end{equation}
for any $n\in\mathbb N$. Since the point $x$ is positively
Lagrange stable, then without loss of generality we can suppose
that the sequence $\{\pi(t_n,x)\}$ converges. Denote by
$\bar{p}:=\lim\limits_{n\to \infty}\pi(t_n,x)$, then $\bar{p}\in
\omega_{x}$ and, consequently,
\begin{equation}\label{eqS.9.2}
\pi(\tau,\bar{p})=\bar{p}.
\end{equation}
Passing to the limit in (\ref{eqS.9.1}) as $n\to \infty$ and
taking into account (\ref{eqS.9.2}) we obtain $\varepsilon_0\le
0$. The last inequality contradicts the choice $\varepsilon_0$.
The obtained contradiction proves our statement.
\end{proof}

\begin{remark}\label{remS1} The motions possessing the property
(\ref{eqS*}) was studied in the works of K. Cryszka
\cite{Gry_2018} and A. Pelczar \cite{Pel_1985}.
\end{remark}

\begin{definition}\label{defAP1} A point $x\in X$ of dynamical
system $(X,\mathbb T,\pi)$ is said to be:
\begin{enumerate}
\item positively Poisson stable if $x\in \omega_{x}$; \item
positively asymptotically Poisson stable if there exists a
positively Poisson stable point $p\in X$ such that
\begin{equation}\label{eqAP3}
\lim\limits_{t\to \infty}\rho(\pi(t,x),\pi(t,p))=0.
\end{equation}
\end{enumerate}
\end{definition}

\begin{theorem}\label{thAP1} Assume that the point $x\in X$ satisfies
the following conditions:
\begin{enumerate}
\item there exists a positive number $\tau \in \mathbb T$ such
that the relation (\ref{eqS*}) holds;
\item the point $x\in X$ is
positively asymptotically Poisson stable.
\end{enumerate}

Then the point $x$ is asymptotically $\tau$-periodic.
\end{theorem}
\begin{proof} Since the point $x$ is positively asymptotically
Poisson stable, then there exists a positively Poisson stable
point $p$ such that (\ref{eqAP3}) holds. From (\ref{eqAP3}) we
obtain $\omega_{x}=\omega_{p}$. On the other hand the point $p\in
\omega_{p}$ because it is positively Poisson stable. Thus we have
$p\in \omega_{p}=\omega_{x}$. By Theorem \ref{th1.3.9} the point
$p\in \omega_{x}$ is $\tau$-periodic and taking into account
(\ref{eqAP3}) we conclude that the point $x$ is asymptotically
$\tau$-periodic. Theorem is proved.
\end{proof}

\begin{theorem}\label{thLS1}\cite{Che_2020}, \cite{sib} Assume that the point $x\in X$
is positively Lagrange stable, then the following statement hold:
\begin{enumerate}
\item $\omega_{x}\not= \emptyset$; \item $\omega_{x}$ is a compact
subset of $X$; \item the set $\omega_{x}$ is invariant, that is,
$\pi(t,\omega_{x})=\omega_{x}$ for any $t\in \mathbb S_{+}$.
\end{enumerate}
\end{theorem}

Let $\Sigma\subseteq X$ be a compact positively invariant set, $
\varepsilon>0$ and $t>0$.

\begin{definition}\label{defCR1}
The collection $\{ x=x_0,x_1,x_2,\dots,x_k=y; t_0,t_1,\dots,t_k\}$
of points $x_i\in\Sigma$ and the numbers $t_i\in\mathbb T$ such
that $t_i\geq t$ and $\rho(\pi(t_i,x_i),x_{i+1})<\varepsilon \
(i=0,1,\dots,k-1)$ is called (see, for example,
\cite{bro84,BB_80}, \cite[Ch.IV]{Che_2015},
\cite[Ch.VI]{Che_2020}, \cite{Con,Con.1}, \cite{Rob} and the
bibliography therein) a
\emph{$(\varepsilon,t,\pi)$-chain}\index{$(\varepsilon,t,\pi)$-chain}
joining the points $x$ and $y$.
\end{definition}

We denote by $P(\Sigma)$\index{$P(\Sigma)$} the set
\begin{equation}\label{eqCR1}
\{ (x,y) : x,y\in\Sigma, \forall \ \varepsilon >0 \ \forall \ t>0\
\exists \ (\varepsilon, t,\pi)-\mbox{chain joining}\ x\
\mbox{and}\ y\}.\nonumber
\end{equation}
The relation $P(\Sigma)$ is closed, invariant and transitive
\cite{bro84,Con,Hur_1995,Pat,Rob}.

\begin{definition}\label{defCR2}
The point $x\in\Sigma$ is called \emph{chain
recurrent}\index{chain recurrent point} in $\Sigma$ if $(x,x)\in
P(\Sigma)$.
\end{definition}

We denote by $\mathfrak R(\Sigma)$\index{$\mathfrak R(\Sigma)$}
the set of any chain recurrent (in $\Sigma$) points from $\Sigma$.

\begin{definition}\label{defCR0} Let $A\subseteq X$
be a nonempty positively invariant set. The set $A$ is called
(see, for example,\cite{HSZ}) \emph{internally chain
recurrent}\index{internally chain recurrent set} if $\mathfrak
R(A)=A$, and \emph{internally chain transitive}\index{internally
chain transitive set} if the following stronger condition holds:
for any $a,b\in A$ and any $\varepsilon
>0$ and $t>0$, there is an $(\varepsilon,t,\pi)$-chain in A
connecting $a$ and $b$.
\end{definition}



\begin{definition}\label{defIC1} A compact invariant set $A$ is
said to be invariantly (respectively, positively invariantly)
connected if it cannot be decomposed into two disjoint closed
nonempty invariant (respectively, positively invariant) sets.
\end{definition}

\begin{lemma}\label{lCR1}\cite{HSZ} Let $x\in X$ and $\gamma\in
\Phi_{x}$. The following statements hold:
\begin{enumerate}
\item the $\omega$-limit set of positive precompact orbit of point
$x$ is internally chain transitive; \item $\mathfrak
R(\omega_{x})=\omega_{x}$, i.e., $\omega$-limit set of $x$ is
internally chain recurrent; \item $\omega$-limit set of $x$ is
invariantly connected.
\end{enumerate}
\end{lemma}

\begin{definition}\label{def2.1} The triplet $\langle W, \phi ,(Y , \mathbb S, \sigma )\rangle
(\mbox{or  briefly}\quad \phi) $ is said to be a cocycle (see, for
example, \cite{Che_2015} and \cite{Sel_1971}) over $(Y,\mathbb
S,\sigma )$ with the fiber $W$ if the mapping $\phi : \mathbb
S_{+} \times Y \times W \to W $ satisfies the following
conditions:
\begin{enumerate}
\item $\phi (0,y,u)=u $ for all $u\in W$ and $y\in Y$; \item $\phi
(t+\tau ,y ,u)=\phi (t,\phi (\tau ,u,y),\sigma(\tau,y))$ for all $
t, \tau \in \mathbb T _{+},u \in W$ and $ y \in Y$; \item the
mapping $\phi $ is continuous.
\end{enumerate}
\end{definition}

Everywhere in this paper we suppose that $W$ is one-dimensional,
i.e., $W$ is a subset of $\mathbb R$ (or it is homeomorphic to a
subset $I$ of $\mathbb R$).

\begin{definition}\label{defC1} A cocycle $ \langle W ,
\phi,(Y,\mathbb S,\sigma)\rangle $ is said to be monotone (order
preserving) if for any $u_1,u_2\in W$ with $u_1\le u_2$ we have
$\varphi(t,u_1,y)\le \varphi(t,u_2,y)$ for any $t\in \mathbb
S_{+}$ and $y\in Y$.
\end{definition}

\begin{definition}\label{def2.2} Let $ \langle W ,
\phi,(Y,\mathbb S,\sigma)\rangle $ be a cocycle on $ W, X:=W\times
Y$ and $\pi $ be a mapping from $ \mathbb S_{+} \times X $ to $X$
defined by the equality $\pi =(\phi ,\sigma)$, i.e., $\pi
(t,(u,y))=(\phi (t,\omega,u),\sigma(t,y))$ for any $ t\in \mathbb
S_{+}$ and $(u,y)\in E\times Y$. The triplet $ (X,\mathbb S_{+},
\pi)$ is an autonomous dynamical system and it is called
\cite{Sel_1971} a skew-product dynamical system.
\end{definition}

\begin{definition}\label{def2.3} Let $(X,\mathbb S_{+},\pi ) $ and $(Y ,\mathbb
S, \sigma )$ be two autonomous dynamical systems and $h: X \to Y$
be a homomorphism from $(X,\mathbb S_{+},\pi )$ to $(Y , \mathbb
S,\sigma)$ (i.e., $h(\pi(t,x))$ $=$ $\sigma(t,h(x)) $ for any
$t\in \mathbb S_{+} $, $ x \in X $ and $h $ is continuous), then
the triplet $\langle (X,\mathbb S_{+},\pi ),$ $(Y ,\mathbb
S,\sigma ),h \rangle $ is called (see \cite{bro75} and
\cite{Che_2015}) a nonautonomous dynamical system (shortly NDS).
\end{definition}

\begin{example} \label{ex2.4} Let $\langle W, \phi ,(Y ,\mathbb S,\sigma)\rangle $ be a cocycle,
$(X,\mathbb S_{+},\pi ) $ be a skew-product dynamical system
($X=W\times Y, \pi =(\phi,\sigma)$) and $h= pr _{2}: X \to Y ,$
then the triplet
\begin{equation}\label{eqNDS}
\langle (X,\mathbb S_{+},\pi ),(Y ,\mathbb S, \sigma ),h \rangle
\end{equation} is a nonautonomous dynamical system.
\end{example}

\begin{definition}\label{defAB1}
Let $M\subseteq X$ be a compact subset of $X$ such that $h(M)=Y$.
Denote by $M_{y}:=h^{-1}(y)=\{x\in M|\ h(x)=y\}$,
$I_{y}:=pr_{1}(M_{y})\subseteq W$,
\begin{equation}\label{eqIS0.01}
\alpha_{M}(y):=\inf\{u\in I_{y}\}\nonumber
\end{equation}
and
\begin{equation}\label{eq002}
\beta_{M}(y):=\sup\{u\in I_{y}\}.\nonumber
\end{equation}
\end{definition}

\begin{theorem}\label{thIS1}\cite{Che_2023}  Let $\langle W,\varphi,(Y,\mathbb
S,\sigma)\rangle$ be a monotone one-dimensional cocycle. Assume
that the following conditions hold:
\begin{enumerate}
\item there exist a compact subset $M$ of $X$ such that $h(M)=Y$;
\item $M$ is invariant; \item every point $y\in Y$ is stationary
(respectively, $\tau$-periodic).
\end{enumerate}

Then $x:=(\gamma(y),y)$ is a stationary (respectively,
$\tau$-periodic) point, where $\gamma(y)=\alpha_{M}(y)$ or
$\beta_{M}(y)$.
\end{theorem}

\section{$S$-asymptotically $\tau$-periodic motions}\label{S3}

Let $\tau \in \mathbb S_{+}$ be a positive number and denote by
$\mathcal S:=\mathbb S \bigcap \mathbb {Z\tau}$ (respectively, by
$\mathcal S_{+}:=\mathbb S_{+} \bigcap \mathbb {Z_{+}}\tau$),
where $\mathbb T\tau :=\{k\tau|\ k\in \mathbb T\}$ ($\mathbb
T=\mathbb Z$ or $\mathbb Z_{+}$). Consider a non-autonomous
dynamical system $\langle (X,\mathbb S_{+},\pi), (Y,\mathbb
S,\sigma),h \rangle$.

\begin{definition}\label{defDD1} A non-autonomous dynamical system $\langle (X,\mathcal
S_{+},\tilde{\pi}), (Y,\mathcal S,\tilde{\sigma}),h \rangle$ is
called the discretization of $\langle (X,\mathbb S_{+},\pi),
(Y,\mathbb S,\sigma),h \rangle$ with a step $\tau$, if the
following conditions are fulfilled:
\begin{enumerate}
\item $\tilde{\pi}(k,x)=\pi(k\tau,x)$ for any $(k,x)\in \mathbb
Z_{+}\times X$; \item $\tilde{\sigma}(k,y)=\sigma(k\tau,,y)$ for
any $(k,y)\in \mathbb Z\times Y$.
\end{enumerate}
\end{definition}

\begin{lemma}\label{lIS1} Let $\langle W,\varphi,(Y,\mathbb
S,\sigma)\rangle$be a monotone one-dimensional cocycle
\cite{Che_2023}, $(X,$ $\mathbb S_{+},$ $\pi)$ be the skew-product
dynamical system associated by the cocycle $\varphi$ and
(\ref{eqNDS}) be the NDS generated by $\varphi$. Assume that the
following conditions are fulfilled:
\begin{enumerate}
\item[a.] $x_0\in X$ is a point with the precompact
semi-trajectory $\Sigma_{x_0}^{+}$.; \item[b.] the point
$y_0:=h(x_0)$ is asymptotically stationary (respectively,
asymptotically $\tau$-periodic), i.e., there exists a stationary
(respectively, $\tau$-periodic) point $q\in \omega_{y_0}$ such
that
\begin{equation}\label{eqAPQ1}
\lim\limits_{t\to
+\infty}\rho(\sigma(t,y_0),\sigma(t,q))=0.\nonumber
\end{equation}
\end{enumerate}

Then
\begin{equation}\label{eqAPQ0}
\tilde{\omega}_{x_0}=\omega_{x_0}\bigcap X_{q},\nonumber
\end{equation}
where $X_{q}=h^{-1}(q)$ and $\tilde{\omega}_{x_0}$ is the
$\omega$-limit set of the point $x_0\in X$ in $(X,\mathcal
S_{+},\tilde{\pi})$.
\end{lemma}
\begin{proof} Let $x\in \tilde{\omega}_{x_0}$, then there exists a
sequence $\{k_n\}\subset \mathbb Z_{+}$ such that $k_n\to \infty$
and
\begin{equation}\label{eqAPQ01}
\tilde{\pi}(k_n,x_0)=\pi(k_n\tau,x_0)\to x
\end{equation}
as $n\to \infty$. Since the point $q$ is $\tau$-periodic, then
from the equality (\ref{eqAPQ1}) we obtain $\lim\limits_{n\to
\infty}\sigma(k_n\tau,y_0)=q$ and, consequently, $x\in
\omega_{x_0}\bigcap X_{q}$ because by (\ref{eqAPQ01}) $x\in
\omega_{x_0}$.

Conversely. Let $x\in \omega_{x_0}\bigcap X_{q}$, then $h(x)=q$
and there exists a sequence $\{t_n\}\subset \mathbb S_{+}$ such
that $t_n\to +\infty$ and $\pi(t_n,x_0)\to x$ as $n\to \infty$.
Note that $t_n$ can be presented as follows $t_n=k_n\tau
+\tau_{n}$, where $(\tau_n,k_n)\in [0,\tau)\times \mathbb Z_{+}$
and $k_n\to \infty$ as $n\to \infty$. It easy to see that
\begin{equation}\label{eqAPQ02}
\pi(t_n,x_0)=\pi(k_n\tau
+\tau_n,x_0)=\pi(\tau_n,\pi(k_n\tau,x_0))=\pi(\tau_n,\tilde{\pi}(k_n,x_0))
\end{equation}
for any $n\in \mathbb Z_{+}$. Taking into consideration that
$\{\tau_n\}\subset [0,\tau)$ and the positive Lagrange stability
of the point $x_0$ without loss of generality we can suppose that
the sequences $\{\tau_n\}$ and
$\tilde{\pi}(k_n,x_0)\}=\{\pi(k_n\tau,x_0)\}$ converge. Denote by
$\tau_0 =\lim\limits_{n\to \infty}\tau_n$ and
$\tilde{x}=\lim\limits_{n\to \infty}\tilde{\pi}(k_n,x_0)$. It is
clear that $\tilde{x}\in \tilde{\omega}_{x_0}$. Passing to the
limit in the equality (\ref{eqAPQ02}) as $n\to \infty$ we obtain
\begin{equation}\label{eqAPQ03}
x=\pi(\tau_0,\tilde{x}).
\end{equation}
On the other hand we have
\begin{equation}\label{eqAPQ04}
h(\tilde{x})=\lim\limits_{n\to
\infty}h(\pi(k_n\tau,x_0))=\lim\limits_{n\to
\infty}\sigma(k_n\tau,h(x_0))=\lim\limits_{n\to
\infty}\sigma(k_n\tau,y_0)=q
\end{equation}
and, consequently, $\tilde{x}\in X_{q}$. From the equalities
(\ref{eqAPQ03})-(\ref{eqAPQ04}) we obtain
\begin{equation}\label{eqAPQ05}
q=h(x)=h(\pi(\tau_0,\tilde{x}))=\sigma(\tau_0,h(\tilde{x}))=\sigma(\tau_0,q).\nonumber
\end{equation}
Since $\tau_0\in [0,\tau]$ and $q$ is $\tau$-periodic, then
$\tau_0 =0$ or $\tau$ and, consequently,
$x=\pi(\tau_0,\tilde{x})\in \tilde{\omega}_{x_0}$. Lemma is
completely proved.
\end{proof}

\begin{theorem}\label{thBD4} Let $\langle W,\varphi,(Y,\mathbb
S,\sigma)\rangle$be a monotone one-dimensional cocycle, $(X,$
$\mathbb S_{+},$ $\pi)$ be the skew-product dynamical system
associated by cocycle $\varphi$ and (\ref{eqNDS}) be the NDS
generated by $\varphi$. Assume that the following conditions are
fulfilled:
\begin{enumerate}
\item $(X,\mathcal S_{+},\tilde{\pi})$ is strictly monotone, i.e.,
$x_1<x_2$ ($x_i=(u_i,y)$, $i=1,2$ and $u_1<u_2$) implies
$\tilde{\pi}(k,x_1)<\tilde{\pi}(k,x_2)$ for any $k\in \mathbb N$;
\item $x_0\in X$ is a point with the pre-compact semi-trajectory
$\Sigma_{x_0}^{+}$; \item the point $y_0:=h(x_0)$ is
asymptotically stationary (respectively, asymptotically
$\tau$-periodic) points of $(X,\mathbb S_{+},\pi)$.
\end{enumerate}

Then the following statements hold:
\begin{enumerate}
\item the $\omega$-limit set $\omega_{x_0}$ of $x_0$ is a
nonempty, compact and invariant set of the skew-product dynamical
system $(X,\mathbb S_{+},\pi)$; \item
$h(\omega_{x_0})=\omega_{y_0}$; \item
$P\tilde{\omega}_{x_0}=\tilde{\omega}_{x_0}$, where
$\tilde{\omega}_{x_0}:=\omega_{x_0}\bigcap X_{q}$,
$P:=\pi(\tau,\cdot)$ and $q=\lim\limits_{k\to \infty}P^{k}y_0$;
\item the point $x:=(\gamma(q),q)$ is stationary (respectively,
$\tau$-periodic), where $\gamma(q)=\alpha_{M}(q)$ or
$\beta_{M}(q)$ and $M:=\tilde{\omega}_{x_0}$; \item the set
$\tilde{\omega}_{x_0}$ consists of stationary (respectively,
$\tau$-periodic) points.
\end{enumerate}
\end{theorem}
\begin{proof}
The first three statements of Theorem are well-known (see, for
example, \cite{Che_2020} and \cite{sib}). The forth statement
follows from Theorem \ref{thIS1}.

Now we will prove the fifth statement.

First step. Every point $x\in \tilde{\omega}_{x_0}$ is
asymptotically stationary, i.e., there exists a stationary point
$p_{x}\in \tilde{\omega}_{x_0}$ such that $P^{k}(x)\to p_{x}$ as
$k\to \infty$. Indeed, since
$P(\tilde{\omega}_{x_0})=\tilde{\omega}_{x_0}$ and the mapping $P:
\tilde{\omega}_{x_0}\to \tilde{\omega}_{x_0}$ is monotone, then
the sequence $\{P^{k}x\}$ is monotone and bounded and,
consequently, it converges. Denote by $p_x:=\lim\limits_{k\to
\infty}P^{k}x$. It easy to see that $P(p_x)=p_x$ and by Theorem
\ref{thAAP1} the point $x$ is asymptotically $\tau$-periodic.

Second step. Let $M:=\tilde{\omega}_{x_0}$. If
$\alpha_{M}(q)=\beta_{M}(q)$, then the required statement is
proved.


Third step. If $\mathcal M\subseteq M$ is a minimal subset of
$(X,\mathcal S_{+},\tilde{\pi})$, then $\mathcal M$ consists of a
single point because  $\alpha_{\mathcal M}(q)$ and
$\beta_{\mathcal M}(q)$ ($\alpha_{\mathcal M}(q), \beta_{\mathcal
M}(q)\in \mathcal M$) are two stationary points of $(X,\mathcal
S_{+},\tilde{\pi})$ and $\mathcal M =
\overline{\{\tilde{\pi}(k,\alpha_{\mathcal M}(q))|\ k\in \mathbb
Z_{+}\}}=\{\alpha_{\mathcal M}(q)\}$.

Fourth step. The set $M$ consists of an infinite number of
different points. If we assume that it is not true, then $M=\{p_1,
p_2,\ldots,p_m\}$ ($m\ge 2$). Note that $P(M)=M,$ where
$P:=\tilde{\pi}(1,\dot)=\pi(\tau,\cdot)$. It easy to check that
every point $p_{j}$ ($j=1,\ldots,m$) is periodic, i.e., there
exists a number $m_{j}\in \mathbb N$ such that
$P^{m_{j}}(p_{j})=p_{j}$. According to the first step $m_{j}=1$
for any $j=1,\ldots,m$ and, consequently, $M$ consists of a finite
numbers of stationary points $\{p_1, p_2,\ldots, p_m\}$. The
latter fact contradicts to the invariant connectivity of the set
$M$. The obtained contradiction proves our statement.

Fifth step. If $p\in M$ is a stationary point of $(X,\mathcal
S_{+},\pi)$, then the point $p$ cannot be isolate. Indeed, if we
suppose that it is not true, then $M=M_1\cup M_2$ ($M_1:=\{p\}$
and $M_2:=M\setminus \{p\}$). Since $p$ is an isolated stationary
point and the map $P:=\tilde{\pi}(1,\cdot)$ is strictly monotone,
then the sets $M_1$ and $M_2$ are closed, invariant and $M_1\cap
M_2=\emptyset$. But this contradicts to the invariant
connectedness of $M$. The obtained contradiction show that each
stationary point $p$ from $M$ cannot be isolated.

Sixth step. Let $x\in M$ be a nonstationary point. Without loss of
the generality we can suppose that the sequence
$\{P^{k}(x)\}_{k\in \mathbb N}$ is strictly decreasing. Since
$\alpha_{M}(q) <P^{k}(x) < \beta_{M}(q)$ for any $k\in \mathbb N$,
then $\{P^{k}(x)\}_{k\in \mathbb N}$ converges. Denote by $p_x$
its limit, then
\begin{enumerate}
\item $P(p_x)=p_x$ ; \item
\begin{equation}\label{eqP0}
\alpha_{M}(q)\le p_{x}< \ldots <P^{k}(x)<P^{k-1}(x)<\ldots
<P(x)<x<\beta_{M}(q);\nonumber
\end{equation}
\item
\begin{equation}\label{eqP01}
\lim\limits_{k\to \infty}\sup\limits_{p_x\le \bar{x}\le
x}|P^{k}(\bar{x})-p_x|=0.\nonumber
\end{equation}
\end{enumerate}
The last equality follows from the inequality
\begin{equation}\label{eqP02}
p_x\le P^{k}(\bar{x})\le P^{k}(x)\nonumber
\end{equation}
for any $k\in \mathbb N$.

We will prove that the point $x$ is not chain recurrent. Let $m\in
\mathbb N$. We shall choose an $\varepsilon > 0$ so that there is
no $(\varepsilon, m)$-chain from $x$ to itself. First, we will
show that there exists $\delta
>0$ such that if $p_{x}<\bar{x} < P^{m}(x) +\delta,$ then
\begin{equation}\label{eqP1}
p_{x}< P^{k}(\bar{x}) < P^{m}(x)\nonumber
\end{equation}
for any $k\ge m$.

Let $\delta :=x-P^{m}(x) >0$. If $p_x<\bar{x} < P^{m}(x) +\delta
=x,$ then
\begin{equation}\label{eqP2}
p_{x}< P^{k}(\bar{x})<P^{k}(x)\le P^{m}(x)\nonumber
\end{equation}
for any $k\ge m$.

Now we will choose an $\varepsilon >0$ so that there is no
$(\varepsilon, m)$-chain from $x$ to itself. Let $\varepsilon
:=\min\{\delta, x-P^{m}(x)\}$. By means of contradiction, assume
that there exists an $(\varepsilon, m)$-chain $(x = x_0,
x_1,\ldots, x_n = x; m_0, m_1,\ldots, m_n)$ ($x_j\in M$ for any
$j=0,1,\ldots,n$) from $x$ to itself, then
\begin{equation}\label{eqP3}
|P^{m_0}(x_0)- x_1 | < \varepsilon \le \delta .\nonumber
\end{equation}
Since $m_0\ge m$ and $\{P^{k}(x_0)\}_{k\in \mathbb N}$ is a
decreasing sequence of $k$,
\begin{equation}\label{eqP4}
x_1 < P^{m_0}(x_0)+ \delta \le P^{m}(x_0) +\delta \le x_0 =x
.\nonumber
\end{equation}
Since $m_1 \ge m,$
\begin{equation}\label{eqP5}
P^{m_1}(x_1)< P^{m_1}(x_0)\le P^{m}(x).\nonumber
\end{equation}
Now, $|P^{m_1}(x_1)-x_2| < \varepsilon \le \delta,$ so that
\begin{equation}\label{eqP6}
x_2 < P^{m_1}(x_1) + \delta < P^{m}(x_0) +\delta \le x_0 =x
.\nonumber
\end{equation}
Thus,
\begin{equation}\label{eqP7}
P^{m_2}(x_2) <  P^{m_2}(x_0)\le P^{m}(x).\nonumber
\end{equation}
Continuing in this manner we obtain
\begin{equation}\label{eqP8}
P^{m_{n-1}}(x_{n-1}) < P^{m}(x) < x.\nonumber
\end{equation}
Since $x_n = x$ and $m_{n-1}\ge m$ we have
\begin{equation}\label{eqP9}
|P^{m_{n-1}}(x_{n-1})-x_n| = |P^{m_{n-1}}(x_{n-1})-x|> x-P^{m}(x).
\end{equation}

On the other hand, by the definition of $\varepsilon$,
\begin{equation}\label{eqP10}
|P^{m_{n-1}}(x_{n-1})-x_n| < \varepsilon \le  x-P^{m}(x).
\end{equation}
The inequalities (\ref{eqP9}) and (\ref{eqP10}) are contradictory.
The obtained contradiction proves our statement. Theorem is
completely proved.
\end{proof}

\begin{remark}\label{remP1} Note that under the conditions of
Theorem \ref{thBD4} the set $M$ can contain an infinite number of
points. In the next section we give an example confirming above.
\end{remark}

\section{Example}\label{S4}

Denote by $C(\mathbb R_{+},\mathbb R)$ the space of all continuous
functions $\varphi :\mathbb R_{+}\to \mathbb R$ equipped with the
compact-open topology and
\begin{equation}\label{eqE1}
\delta_{\varphi}:=\bigcap_{t\ge 0}\overline{\bigcup _{\tau \ge
t}\varphi(\tau)} .\nonumber
\end{equation}

\begin{lemma}\label{lE1} The following statements hold:
\begin{enumerate}
\item $x\in \delta_{\varphi}$ if and only if there exists a
sequence $\{t_n\}\subset \mathbb R_{+}$ such that $t_n\to +\infty$
and $\varphi(t_n)\to x$ as $n\to \infty$; \item if the function
$\varphi \in C(\mathbb R_{+},\mathbb R)$ is bounded, then:
\begin{enumerate}
\item
$\delta_{\varphi} \not= \emptyset$; \item $\delta_{\varphi}$ is a
compact subset of $\mathbb R$; \item
\begin{equation}\label{eqE2}
\lim\limits_{t\to +\infty}\rho(\varphi(t),\delta_{\varphi})=0,
\nonumber
\end{equation}
where $\rho(x,M):=\inf\limits_{y\in M}|x-y|$ and $M$ is a compact
subset of $\mathbb R$.
\end{enumerate}
\end{enumerate}
\end{lemma}
\begin{proof} This statement can be proved by the standards
arguments using in the theory of dynamical systems (see, for
example, \cite[Ch.II]{sib}).
\end{proof}

\begin{lemma}\label{lE2} Assume that $\varphi \in C(\mathbb R_{+},\mathbb
R)$ is a bounded function, then $\delta_{\varphi}$ is a connected
compact subset of $\mathbb R$, i.e., there exist $\alpha$ and
$\beta$ from $\mathbb R$ such that
\begin{equation}\label{eqE3}
\delta_{\varphi}=[\alpha,\beta]:=\{x\in \mathbb R|\ \alpha \le
x\le \beta\}.\nonumber
\end{equation}
\end{lemma}
\begin{proof} Assume that this statement is false, then there
are two nonempty compact subsets $K_1,K_2\subset \mathbb R$ such
that
\begin{equation}\label{eqE4}
\delta_{\varphi}=K_1\bigcup K_2 \ \ \mbox{and}\ \ \ K_1\bigcap K_2
=\emptyset .
\end{equation}
Denote by
\begin{equation}\label{eqE5}
d :=\inf\limits_{a\in K_1,b\in K_2}|a-b| ,\nonumber
\end{equation}
then under the condition (\ref{eqE4}) the number $d >0$. Let now
$\varepsilon \in (0,d/3)$, $p_i\in K_{i}$ ($i=1,2$), $p_1<p_2$
($p_2-p_1\ge d>0$). For given $\varepsilon$ and $p_1\in K_{1}$
there exists $t_1>1$ such that $|\varphi(t_1)-p_1|<1$. For
$\varepsilon$ and $p_2\in K_2$ there exists $t_2>2$ such that
$|\varphi(t_2)-p_2|<1/2$. Further there exist $t_3>3$ ($t_3>t_2$)
such that $|\varphi(t_3)-p_1|<1/3$ and $t_4>4$ ($t_4>t_3$) with
$|\varphi(t_4)-p_2|<1/4$ and so on. Thus we obtain a sequence
$t_1<t_2<\ldots <t_n<\ldots $ possessing with the properties:
\begin{enumerate}
\item \begin{equation}\label{eqE5.1} | \varphi(t_{2k+1})-p_1|<
\frac{1}{2k+1} \ \ \mbox{and}\ \ \
|\varphi(t_{2k+2})-p_2|<\frac{1}{2k+2}
\end{equation}
for any $k\in \mathbb Z_{+}$; \item $t_k>k$ for any $k\in \mathbb
N$.
\end{enumerate}

Passing to the limit in (\ref{eqE5.1}) as $k\to \infty$ we obtain
$$
\lim\limits_{k\to \infty}\varphi(t_{2k+1})= p_1 \ \ \mbox{and}\ \
\ \lim\limits_{k\to \infty}\varphi(t_{2k+2}) = p_2.
$$
Assume that
$k_{\varepsilon}\in \mathbb N$ so that $1/k<\varepsilon$ for any
$k>k_{\varepsilon}$, then we have
$$
p_1-\varepsilon \le \varphi(t_{2k+1})\le p_1+\varepsilon \ \
\mbox{and}\ \ p_2 -\varepsilon \le \varphi(t_{2k+2})\le
p_2+\varepsilon .
$$
Denote by $p:=(p_1+p_2)/2$ and
$\psi(t):=\varphi(t)-p$ for any $t\in \mathbb R_{+}$. Note that
$$
\psi(t_{2k+1})<-d/6<0 \ \ \mbox{and}\ \ \psi(t_{2k+2})>d/6 >0
$$
for
any $k\ge k_{\varepsilon}$ and, consequently, there exists
$t_{2k+1}<\bar{t}_{k}<t_{2k+2}$ such that
$\varphi(\bar{t}_{k})=p$. This means that $\varphi(\bar{t}_{k})\to
p$ as $k\to \infty$ and, consequently, $p\in \delta_{\varphi}$
(because $\bar{t}_{k}>t_{2k+1}>2k+1\to \infty$ as $k\to\infty$).

On the other hand $p_1<p<p_2$ implies that $p\notin K_1\bigcup
K_2=\delta_{\varphi}$. The obtained contradiction proves our
statement. Taking into account that $\delta_{\varphi}$ is a
nonempty, compact and connected subset of $\mathbb R$, then there
exist $\alpha$ and $\beta$ from $\mathbb R$ such that
$\delta_{\varphi}=[\alpha,\beta]$. Lemma is completely proved.
\end{proof}

\begin{coro}\label{corE1} Assume that $\varphi \in C(\mathbb R_{+},\mathbb
R)$ is a bounded function, then $\delta_{\varphi}=[\alpha,\beta]$,
where $\alpha =\sup\{x|\ x\in \delta_{\varphi}\}=\max\{x|\ x\in
\delta_{\varphi}\}$ and $\beta =\inf\{x|\ x\in
\delta_{\varphi}\}=\min\{x|\ x\in \delta_{\varphi}\}$.
\end{coro}

\begin{lemma}\label{lE2.1} Suppose that $\varphi \in C(\mathbb R_{+},\mathbb
R)$ is a bounded function, then
$\delta_{\varphi}=[\bar{\alpha},\bar{\beta}]$, where
\begin{equation}\label{eqE8}
\bar{\alpha} =\limsup \limits_{t\to +\infty}\varphi(t)\ \
\mbox{and}\ \ \ \bar{\beta} =\liminf\limits_{t\to
+\infty}\varphi(t).
\end{equation}
\end{lemma}
\begin{proof}
Since the function $\varphi \in C(\mathbb R_{+},\mathbb R)$ is
bounded, then by (\ref{eqE8}) are well defined two real numbers,
i.e., $\bar{\alpha},\bar{\beta} \in \mathbb R$.

Note that for $\bar{\alpha}$ (respectively, $\bar{\beta}$) and
arbitrary $\varepsilon >0$ there exists a positive number
$L(\varepsilon)>0$ such that
\begin{equation}\label{eqE9}
\bar{\beta}-\varepsilon \le \varphi(t) \le \bar{\alpha}
+\varepsilon
\end{equation}
for any $t\ge L(\varepsilon)$ and there exists a sequence
$t^{\bar{\alpha}}_{n}\to +\infty$ (respectively,
$t^{\bar{\beta}}_{n}\to +\infty$) such that
\begin{equation}\label{eqE10}
\varphi(t_{n}^{\bar{\alpha}})\to \bar{\beta}\ \
\mbox{(respectively,}\ \varphi(t_{n}^{\bar{\beta}})\to
\bar{\beta})\nonumber
\end{equation}
and, consequently, $\bar{\alpha},\bar{\beta}\in \delta_{\varphi}$
($\beta \le \bar{\beta}\le \bar{\alpha} \le \alpha$). To finish
the proof of Lemma it suffices to show that $\bar{\alpha}=\alpha$
(respectively, $\bar{\beta}=\beta$). We will show the equality
$\bar{\alpha}=\alpha$ because the equality $\bar{\beta}=\beta$ may
be  proved using the same arguments.

Assume that $\alpha <\bar{\alpha}$ and $0<\varepsilon
<\frac{\bar{\alpha}-\alpha}{3}$. By Corollary \ref{corE1} there
exists a sequence $t_n\to +\infty$ such that $\varphi(t_n)\to
\alpha$ as $n\to \infty$. Since $t_n\to +\infty$ as $n\to \infty$,
then there exists $n_{\varepsilon}\in \mathbb N$ such that $t_n\ge
L(\varepsilon)$ for any $n\ge n_{\varepsilon}$ and by (\ref{eqE9})
we have
\begin{equation}\label{eqE11}
\bar{\alpha} -\varepsilon \le \varphi(t_n)
\end{equation}
for any $n\ge n_{\varepsilon}$. Passing to the limit in
(\ref{eqE11}) as $n\to \infty$ we obtain $\bar{\alpha}-\varepsilon
\le \alpha$, i.e., $\bar{\alpha}-\alpha \le \varepsilon$. This
relation contradicts to the choice of $\varepsilon \in
(0,\frac{\bar{\alpha}-\alpha}{3})$. The obtained contradiction
proves our statement.
\end{proof}

\begin{example}\label{exP1}
Let $a\in C(\mathbb R_{+},\mathbb R)$ defined by the equality
\begin{equation}\label{eqP101}
a(t)=\frac{1}{2\sqrt{\pi^2+t}}\cos{\sqrt{\pi^2 +t}}\nonumber
\end{equation}
for any $t\in \mathbb R_{+}$. Then
\begin{enumerate}
\item
\begin{equation}\label{eqP12}
\lim\limits_{t\to +\infty}|a(t)|=0;\nonumber
\end{equation}
\item $|\varphi(t)|\le 1$ for any $t\in \mathbb R_{+}$, where
\begin{equation}\label{eqP103}
\varphi(t)=\int_{0}^{t}a(s)ds=\sin{\sqrt{\pi^2 +t}} ;\nonumber
\end{equation}
 \item the set $\{\varphi^{h}|\ h\in \mathbb R_{+}\}$ is precompact
in $C(\mathbb R_{+},\mathbb R_{+})$; \item $\psi'(t)=0$ for any
$\psi \in \omega_{\varphi}$ and $t\in \mathbb R$; \item
\begin{equation}\label{eqP13.01}
\delta_{\varphi}=[-1,1].
\end{equation}
\end{enumerate}

The first four statements are evident. To show the equality
(\ref{eqP13.01}) we note that $\delta_{\varphi}\subseteq [-1,1]$
and, consequently, it suffices to establish that $-1,1\in
\delta_{\varphi}$. We will show that $-1\in \delta_{\varphi}$
because the inclusion $1\in \delta_{\varphi}$ may be proved
similarly. Foe any $n\in \mathbb N$ denote by $t_n$ a solution of
the equation
\begin{equation}\label{eq13.2}
\sin \sqrt{\pi^2+t}=-1+\frac{1}{n}\nonumber
\end{equation}
with $t_n\ge n$, then $t_n\to +\infty$ and $\varphi(t_n)\to -1$ as
$n\to \infty$, i.e., $-1\in \delta_{\varphi}$

Consider a differential equation
\begin{equation}\label{eqP14}
x'=a(t).
\end{equation}
Along with the equation (\ref{eqP14}) consider its $H^{+}$-class,
i.e., the family of equations
\begin{equation}\label{eqP15}
y'=b(t),
\end{equation}
where $b\in H^{+}(a):=\overline{\{a^{h}|\ h\in \mathbb R_{+}\}}$,
$a^{h}$ is the $h$-translation of $a\in C(\mathbb R_{+},\mathbb
R)$ and by bar the closure in the space $C(\mathbb R_{+},\mathbb
R)$ is denoted. Since $a(t)\to 0$ as $t\to +\infty$, then the
$\omega$-limit set $\omega_{a}$ consists of a unique stationary
point (stationary function) $\theta$, where $\theta (t)=0$ for any
$t\in \mathbb R$. Thus for the equation (\ref{eqP14}) we have a
unique $\omega$-limit equation
\begin{equation}\label{eqP16}
y'=0 .\nonumber
\end{equation}
Denote by $\varphi(t,v,b)$ the unique solution of the equation
(\ref{eqP15}) passing through the point $v\in \mathbb R$ at the
initial moment $t=0$ and defined on $\mathbb R_{+}$. It is clear
that $\varphi(t,v,b)=v+\int_{0}^{t}b(s)ds$. Let $Y:=H^{+}(a)$ and
$(Y,\mathbb R_{+},\sigma)$ be the shift dynamical system on
$H^{+}(a)$. It easy to check that the mapping $\varphi :\mathbb
R_{+}\times \mathbb R\times H^{+}(a)\to \mathbb R$ ($(t,v,b)\to
\varphi(t,v,b)$) is continuous and
\begin{enumerate}
\item $\varphi(0,v,b)=v$ for any $(v,b)\in \mathbb R\times
H^{+}(a)$; \item
$\varphi(t+s,v,b)=\varphi(t,\varphi(s,v,b),b^{s})$ for any $t,s\in
\mathbb R_{+}$ and $(v,b)\in \mathbb R\times H^{+}(a)$.
\end{enumerate}
This means that the triplet $\langle \mathbb R, \varphi,
(Y,\mathbb R_{+},\sigma)\rangle$ is a cocycle over $(Y,\mathbb
R_{+},\sigma)$ with the fibre $\mathbb R$.

Consider the skew-product dynamical system $(X,\mathbb R_{+},\pi)$
($X:=\mathbb R_{+}\times H^{+}(a)$ and $\pi :=(\varphi,\sigma)$)
generated by cocycle $\varphi$. Let $x_0:=(0,a)\in X$, then
$\pi(s,x_0)=(\varphi(s),a^{s})$ for any $s\in \mathbb R_{+}$. It
is clear that $\{\pi(s,x_0)|\ s\in \mathbb R_{+}\}$ is precompact
and by Theorem \ref{thBD4} the set $M=\omega_{x_0}\bigcap X_{q}$
($q=\theta$, $\omega_{a}=\{\theta\}$) consists of a stationary
points. Let $\bar{x}\in M$, then according to properties of the
function $\varphi$ (see item (v)) there exists a number
$\bar{v}\in [-1,1]$ such that $\bar{x}=(\bar{v},\theta)$ (and visa
versa). The required example is constructed.
\end{example}

\section{Asymptotically $\tau$-periodic motions}\label{S5}

We suppose in this section that the point $y_0\in Y$ is
asymptotically $\tau$-periodic, then there exists
$\lim\limits_{k\to \infty}\sigma(k\tau,y_0)=q$. As we said (see
Example \ref{exP1}) under the conditions of Theorem \ref{thBD4}
the point $x_0$ generally speaking is not an asymptotically
$\tau$-periodic point. Below we will indicate some general
(sufficient) conditions which assure the asymptotically
$\tau$-periodicity of the point $x_0$.

Let $x_0=(u_0,q)\in X=W\times Y$ be a $\tau$-periodic point of the
monotone one-dimensional cocycle $\langle W,\varphi,(Y,\mathbb
T,\sigma)\rangle$, i.e., $\pi(\tau,x_0)=x_0$ (or equivalently,
$\varphi(\tau,u_0,q)=u_0$ and $\sigma(\sigma,q)=q$).

\begin{definition}\label{defIP1} The $\tau$-periodic point $x_0=(u_0,q)$
 is called:
\begin{enumerate}
\item  isolated if there is a positive number $\delta$ such that
the segment $(u_0-\delta,u_0+\delta)\subseteq W$ does not contain
points $u$ other than $u_0$ such that the point $x=(u,q)$ is
$\tau$-periodic; \item positively Lyapunov stable if for any
$\varepsilon
>0$ there exists a positive number $\delta=\delta(\varepsilon)$
such that $\rho(u,u_0)<\delta$ implies
$\rho(\varphi(t,u,q),\varphi(t,u_0,q))<\varepsilon$ for any $t\ge
0$; \item positively attracting if there exists a positive number
$\gamma$ such that
\begin{equation}\label{eqIP1}
\lim\limits_{t\to\infty}\rho(\varphi(t,u,q),\varphi(t,u_0,q))=0
\end{equation}
for any $u\in (u_0-\gamma,u_0+\gamma)$, where $\rho(u,v):=|u-v|$
($u,v\in \mathbb R$); \item positively asymptotically stable if it
is positively Lyapunov stable and positively attracting.
\end{enumerate}
\end{definition}

\begin{theorem}\label{thIP1} Let $x_0=(u_0,q)\in X=W\times Y$ be a $\tau$-periodic point of the
monotone one-dimensional cocycle $\langle W,\varphi,(Y,\mathbb
T,\sigma)\rangle$. Then the following statements hold:
\begin{enumerate}
\item if $x_0$ is positively attracting, then it is isolated;
\item if $x_0$ is positively Lyapunov stable and isolated, then it
is positively asymptotically stable.
\end{enumerate}
\end{theorem}
\begin{proof} To prove the first statement assume that $x_0$ is
positively asymptotically stable but it is not isolated, then
there exists a sequence $x_k=(u_k,q)$ of pairwise different
$\tau$-periodic point of the cocycle $\varphi$ such that $x_k\to
x_0$ as $k\to\infty$. Let $\gamma$ be a positive number such that
(\ref{eqIP1}) holds for any $u\in (u_0-\gamma,u_0+\gamma)$. Now we
choose $k_0\in \mathbb N$ such that $u_k\in
(u_0-\gamma,u_0+\gamma)$ for any $k\ge k_0$. Thus we have
$$\lim\limits_{t\to\infty}\rho(\varphi(t,u_k,q),\varphi(t,u_0,q))=0$$
for any $k\ge k_0$ and, consequently, $u_k=u_0$ for any $k\ge 0$.
The last relation contradicts to our assumption. The obtained
contradiction proves the required statement.

Let now $x_0$ be positively Lyapunov stable and $\gamma >0$ be a
positive number so that the segment $(u_0-\gamma, u_0+\gamma)$
does not contain points $u$ other than $u_0$ such that the point
$x=(u,q)$ is $\tau$-periodic. Now we choose a positive number
$\delta_0 =\delta(\gamma/2)$ from the positive Lyapunov stability
of $x_0$, then $\rho(u,u_0)<\delta_0$ implies
\begin{equation}\label{eqIP3}
\rho(\varphi(t,u,q),\varphi(t,u_0,q))<\gamma/2
\end{equation}
for any $t\ge 0$.

We will show that
\begin{equation}\label{eqIP4}
\lim\limits_{t\to
+\infty}\rho(\varphi(t,u,q),\varphi(t,u_0,q))=0\nonumber
\end{equation}
for any $u\in (u_0-\delta_0,u_0+\delta_0)$. Indeed, from
(\ref{eqIP3}) it follows that the point $x:=(u,q)$ is positively
Lagrange stable. Since the one-dimensional cocycle $\varphi$ is
monotone, then reasoning as in the proof of Theorem \ref{thBD4} we
prove that the point $x$ is asymptotically $\tau$-periodic, i.e.,
there exists a $\tau$-periodic point $p_x$ such that
\begin{equation}\label{eqIP5}
\lim\limits_{t\to +\infty}\rho(\pi(t,x),\pi(t,p_x))=0.\nonumber
\end{equation}
In particular, there exists a number $m\in \mathbb N$ such that
\begin{equation}\label{eqIP6}
\rho(\pi(m\tau,x),p_x)<\gamma/2 .
\end{equation}
Then taking into account (\ref{eqIP3}) and (\ref{eqIP6}) we obtain
\begin{equation}\label{eqIP7}
\rho(p_x,x_0)\le
\rho(p_x,\pi(m\tau,x))+\rho(\pi(m\tau,x),x_0)<\gamma/2+\gamma/2=\gamma,\nonumber
\end{equation}
and, consequently, $p_x=x_0$. This means that
\begin{equation}\label{eqIP8}
\lim\limits_{t\to
+\infty}\rho(\varphi(t,u,q),\varphi(t,u_0,q))=0,\nonumber
\end{equation}
i.e., $x_0=(u_0,q)$ is an attracting point. Theorem is completely
proved.
\end{proof}

\begin{theorem}\label{thBD5} Let $\langle W,\varphi,(Y,\mathbb
S,\sigma)\rangle$be a monotone one-dimensional cocycle, $(X,$
$\mathbb S_{+},$ $\pi)$ be the skew-product dynamical system
associated by the cocycle $\varphi$ and (\ref{eqNDS}) be the NDS
generated by $\varphi$. Assume that the following conditions are
fulfilled:
\begin{enumerate}
\item $(X,\mathcal S_{+},\tilde{\pi})$ is strictly monotone, i.e.,
$x_1<x_2$ implies $\tilde{\pi}(k,x_1)<\tilde{\pi}(k,x_2)$ for any
$k\in \mathbb N$; \item $x_0\in X$ is a point with the pre-compact
semi-trajectory $\Sigma_{x_0}^{+}$.; \item the point $y_0:=h(x_0)$
is asymptotically stationary (respectively, asymptotically
$\tau$-periodic); \item the periodic points $(u,q)$ of the
skew-product dynamical system $(X,\mathbb T_{+},\pi)$ are
isolated.
\end{enumerate}

Then the point $x_0$ is asymptotically $\tau$-periodic.
\end{theorem}
\begin{proof} By Theorem \ref{thBD4} the point $x_0$ is $S$-asymptotically
$\tau$-periodic and, consequently, the set
$M:=\tilde{\omega}_{x_0}$ consists of stationary points of the map
$P:=\pi(\tau,\cdot)$. Note that $M$ is a compact subset of $X$.
Since the periodic points of $(X,\mathbb T_{+},\pi)$ are isolated,
then the set $M$ consists of a finite number of points, i.e.,
$M=\{x_1,\ldots, x_m\}$ ($m\ge 1$). If $m\ge 2$, then this
contradicts to invariant connectedness of $M$. The obtained
contradiction show that $m=1$, i.e., the set $M$ consists of a
single point $x=(u,q)$ and, consequently, the sequence
$\{\pi(k\tau,x_0)\}$ converges to $x$ as $k\to \infty$. This means
that the point $x_0$ is asymptotically $\tau$-periodic.
\end{proof}

\begin{theorem}\label{thBD6} Let $\langle W,\varphi,(Y,\mathbb
S,\sigma)\rangle$ be a monotone one-dimensional cocycle, $(X,$
$\mathbb S_{+},$ $\pi)$ be the skew-product dynamical system
associated by the cocycle $\varphi$ and (\ref{eqNDS}) be the NDS
generated by $\varphi$. Assume that the following conditions are
fulfilled:
\begin{enumerate}
\item $(X,\mathcal S_{+},\tilde{\pi})$ is strictly monotone \item
$x_0\in X$ is positively Lagrange stable; \item the point
$y_0:=h(x_0)$ is asymptotically stationary (respectively,
asymptotically $\tau$-periodic); \item the periodic points $(u,q)$
of the skew-product dynamical system $(X,\mathbb T_{+},\pi)$ are
positively asymptotically stable.
\end{enumerate}

Then the point $x_0$ is asymptotically $\tau$-periodic.
\end{theorem}
\begin{proof} By Theorem \ref{thIP1} (item (i)) the $\tau$-periodic points $x_0=(u_0,q)$ of the cocycle
$\varphi$ are isolated. Now to finish the proof of Theorem it
suffices apply Theorem \ref{thBD5}.
\end{proof}

Assume that the base dynamical system consists from a
$\tau$-periodic trajectory, i.e., $Y=\{\sigma(t,q)|\ t\in
[0,\tau)\}$ and the skew-product dynamical system $(X,\mathbb
T,\sigma)$ generated by cocycle $\langle W,\varphi,(Y,\mathbb
T,\sigma)\rangle $ is two-sided, i.e., $\mathbb T=\mathbb S$
($\mathbb S=\mathbb R$ or $\mathbb Z$).

Denote by $\hat{\sigma}$ (respectively, $\hat{\pi}$ or
$\hat{\varphi}$) the mapping $\hat{\sigma} :\mathbb S\times Y\to
Y$ (respectively, $\hat{\varphi}:\mathbb S\times W\times Y\to W$
or $\hat{\pi}:=(\hat{\varphi},\hat{\sigma})$) defined by equality
$\hat{\sigma}(t,y):=\sigma(-t,y)$ for any $(t,y)\in \mathbb
S\times Y$ (respectively, $\hat{\varphi}(t,u,y):=\varphi(-t,u,y)$
for any $(t,u,y)\in \mathbb S\times W\times Y$).

It easy to check that the triplet $(Y,\mathbb S,\hat{\sigma})$
(respectively, the cocycle $\langle W,\hat{\varphi},(Y,\mathbb
T,\hat{\sigma})\rangle $ or $(X,\mathbb S,\hat{\pi})$) is a
dynamical system on $Y$ (respectively, $\hat{\varphi}$ is a
cocycle over $(Y,\mathbb S,\hat{\sigma})$ with the fibre $W$ and
$(X,\mathbb S,\hat{\pi})$ is the skew-product dynamical system
associated by the cocycle $\hat{\varphi}$).

\begin{definition}\label{defIP2} A $\tau$-periodic point $x=(u,q)$ is said
to be negatively Lyapunov stable (respectively, negatively
attracting or negatively asymptotically stable) if it is
positively Lyapunov stable (respectively, positively attracting or
positively asymptotically stable) with respect to the
nonautonomous (cocycle) dynamical system $\langle
W,\hat{\varphi},(Y,\mathbb S,\hat{\sigma})\rangle$).
\end{definition}

\begin{remark}\label{remIP3} Note that Theorems \ref{thIP1}
and \ref{thBD6} remains true for the two-sided one-dimensional
nonautonomous dynamical systems if we replace everywhere positive
stability (respectively, positive attraction or positive
asymptotic stability) by negative stability (respectively,
negative attraction or negative asymptotic stability).
\end{remark}

\begin{definition}\label{defIP4} A $\tau$-periodic point $x=(u,q)$
of the two-sided nonautonomous dynamical system $\langle
W,\varphi,(Y,\mathbb S,\sigma)\rangle$ is said to be attractive
(respectively, Lyapunov stable or asymptotically stable) if it is
attractive (respectively, Lyapunov stable or asymptotically
stable) in the positive or negative direction.
\end{definition}

\begin{coro}\label{corIP1} Let $x_0=(u_0,q)\in X=W\times Y$ be a $\tau$-periodic point of the
monotone one-dimensional two-sided cocycle $\langle
W,\varphi,(Y,\mathbb S,\sigma)\rangle$. Then the following
statements hold:
\begin{enumerate}
\item if $x_0$ is attracting (positively or negatively), then it
is isolated; \item if $x_0$ is Lyapunov stable (positively or
negatively) and isolated, then it is asymptotically stable
(positively or negatively, respectively).
\end{enumerate}
\end{coro}
\begin{proof} This statement follows from Theorem \ref{thIP1} and
Remark \ref{remIP3}.
\end{proof}

\begin{coro}\label{corBD6} Let $\langle W,\varphi,(Y,\mathbb
S,\sigma)\rangle$ be a monotone one-dimensional two-sided cocycle,
$(X,$ $\mathbb S_{+},$ $\pi)$ be the skew-product dynamical system
associated by the cocycle $\varphi$ and (\ref{eqNDS}) be the NDS
generated by $\varphi$. Assume that the following conditions are
fulfilled:
\begin{enumerate}
\item $(X,\mathcal S_{+},\tilde{\pi})$ is strictly monotone \item
$x_0\in X$ is positively Lagrange stable; \item the point
$y_0:=h(x_0)$ is asymptotically stationary (respectively,
asymptotically $\tau$-periodic); \item every periodic point
$(u,q)$ of the skew-product dynamical system $(X,\mathbb
T_{+},\pi)$ is asymptotically stable (positively or negatively).
\end{enumerate}

Then the point $x_0$ is asymptotically $\tau$-periodic.
\end{coro}
\begin{proof} This statement follows from Theorem \ref{thBD6} and
Remark \ref{remIP3}
\end{proof}

\section{Asymptotically periodic functions}\label{S6}

Let $(X,\rho)$ be a complete metric space. Denote by $C(\mathbb
T,X)$ the space of all continuous functions $\varphi :\mathbb T\to
X$ equipped with the distance
\begin{equation*}\label{eqD1}
d(\varphi,\psi):=\sup\limits_{L>0}\min\{\max\limits_{|t|\le L,\
t\in \mathbb T}\rho(\varphi(t),\psi(t)),L^{-1}\}.
\end{equation*}
The space $(C(\mathbb T,X),d)$ is a complete metric space (see,
for example, \cite[ChI]{Che_2020}).

\begin{lemma}\label{l1} {\rm(\cite{Sch72}-\cite{sib})} The following statements
hold:
\begin{enumerate}
\item $d(\varphi,\psi) = \varepsilon$ if and only if
$$
\max\limits_{|t|\le
\varepsilon^{-1}}\rho(\varphi(t),\psi(t))=\varepsilon ;
$$
\item $d(\varphi,\psi)<\varepsilon$ if and only if
$$
\max\limits_{|t|\le
\varepsilon^{-1}}\rho(\varphi(t),\psi(t))<\varepsilon ;
$$
\item $d(\varphi,\psi)>\varepsilon$ if and only if
$$
\max\limits_{|t|\le
\varepsilon^{-1}}\rho(\varphi(t),\psi(t))>\varepsilon .
$$
\end{enumerate}
\end{lemma}

\begin{remark}\label{remD1} \rm
1. The distance $d$ generates on the space $C(\mathbb R,X)$ the
compact-open topology.

2. The following statements are equivalent:
\begin{enumerate}
\item $d(\varphi_n,\varphi)\to 0$ as $ n\to \infty$; \item
$\lim\limits_{n\to \infty}\max\limits_{|t|\le
L}\rho(\varphi_n(t),\varphi(t))=0$ for each $L>0$; \item there
exists a sequence $l_n\to +\infty$ such that $\lim\limits_{n\to
\infty}\max\limits_{|t|\le l_n}\rho(\varphi_n(t),\varphi(t))=0$.
\end{enumerate}
\end{remark}

Let $h\in \mathbb T$, $\varphi \in \mathbb C(\mathbb T,X)$ and
$\varphi^{h}$ be the $h$-translation, i.e.,
$\varphi^{h}(t):=\varphi(t+h)$ for any $t\in \mathbb T$. Denote by
$\sigma_{h}$ the mapping from $C(\mathbb T,X)$ into itself defined
by equality $\sigma_{h}\varphi :=\varphi^{h}$ for any $\varphi \in
C(\mathbb T,X)$. Note that $\sigma_{0}=Id_{C(\mathbb T,X)}$ and
$\sigma_{h_1}\sigma_{h_2}=\sigma_{h_1+h_2}$ for any $h_1,h_2\in
\mathbb T$.

\begin{lemma}\label{lAPF1}\cite[Ch.I]{Che_2020} The mapping $\sigma :\mathbb T\times C(\mathbb T,X)\to C(\mathbb
T,X)$ ($\sigma(h,\varphi):=\sigma_{h}\varphi$ for any
$(h,\varphi)\in \mathbb T\times C(\mathbb T,X)$) is continuous.
\end{lemma}

\begin{coro}\label{corAPF1} The triplet $(C(\mathbb T,X),\mathbb
T,\sigma)$ is a dynamical system (shift dynamical system or
Bebutov's dynamical system).
\end{coro}

Let $\mathfrak B$ be a  Banach space with the norm $|\cdot|$,
$\rho(u,v):=|u-v|$ ($u,v\in \mathfrak B$) and $\tau\in \mathbb T$
be a positive number. Denote by $C_{0}(\mathbb T,\mathfrak
B)):=\{\varphi \in C(\mathbb T,\mathfrak B)$ such that
$\lim\limits_{t\to +\infty}|\varphi(t)|=0\}$ and $C_{\tau}(\mathbb
T,\mathfrak B):=\{\varphi(\mathbb T,\mathfrak B)|\
\varphi(t+\tau)=\varphi(t)$ for any $t\in \mathbb T\}$.

\begin{definition}\label{defAPF1} Let $\tau \in \mathbb T$ and $\tau >0$. A function $\varphi \in C(\mathbb
T,X)$is said to be:
\begin{enumerate}
\item asymptotically $\tau$-periodic (respectively, asymptotically
stationary) if there exist $p\in C_{\tau}(\mathbb T,\mathfrak B)$
(respectively, $p$ is a stationary function) and $r\in
C_{0}(\mathbb T,\mathfrak B)$ such that $\varphi(t)=p(t)+r(t)$ for
any $t\in \mathbb T$; \item $S$-asymptotically $\tau$-periodic
\cite{HPT_2008} (respectively, $S$-asymptotically stationary) if
\begin{equation}\label{eqAPF2}
\lim\limits_{t\to
+\infty}\rho(\varphi(t+\tau),\varphi(t))=0\nonumber
\end{equation}
(respectively, $S$-asymptotically $\tau$-periodic for any $\tau
>0$).
\end{enumerate}
\end{definition}

\begin{remark}\label{remAPF1} Every asymptotically $\tau$
periodic (respectively, asymptotically stationary) function
$\varphi \in C(\mathbb T,\mathfrak B)$ is $S$-asymptotically
$\tau$-periodic \cite{HPT_2008} (respectively, $S$-asymptotically
stationary).
\end{remark}

\begin{definition}\label{defAPF2} A function $\varphi \in C(\mathbb
T,X)$ is said to be positively Lagrange stable (respectively,
asymptotically $\tau$-periodic, asymptotically stationary etc) if
the motion $\sigma(t,\varphi)$ is so in the shift dynamical system
$(C(\mathbb T,X),\mathbb T,\sigma)$.
\end{definition}

\begin{lemma}\label{lAPF2}  {\rm(\cite{Sel_1971}-\cite{sib})}  A function $\varphi \in C(\mathbb T,X)$
is positively Lagrange stable if and only if the following
conditions are fulfilled:
\begin{enumerate}
\item the set $\{\varphi(t)|\ t\in \mathbb T,\ t\ge 0\}$ is
precompact in $X$; \item the function $\varphi$ is uniformly
continuous on $\mathbb T_{+}$.
\end{enumerate}
\end{lemma}

\begin{lemma}\label{lAPF3} Let $\varphi \in C(\mathbb T,X)$ and $\tau \in
\mathbb T$ be a positive number. The following statements are
equivalent:
\begin{enumerate}
\item
\begin{equation}\label{eqAPF3}
\lim\limits_{t\to +\infty}\rho(\varphi(t+\tau),\varphi(t))=0;
\end{equation}
\item
\begin{equation}\label{eqAPF4}
\lim\limits_{t\to
+\infty}d(\sigma(t+\tau,\varphi),\sigma(t,\varphi))=0 .
\end{equation}
\end{enumerate}
\end{lemma}
\begin{proof}
We will show that (\ref{eqAPF3}) implies (\ref{eqAPF4}). Assume
that it is not true. Then there are $\varepsilon_0>0$ and $t_m\to
+\infty$ ($t_m\in \mathbb T$) as $m\to \infty$ such that
\begin{equation}\label{eqAPF6}
\rho(\varphi^{t_m+\tau},\varphi^{t_m})\ge \varepsilon_0
\end{equation}
for any $m\in \mathbb N$. By Lemma \ref{l1} (\ref{eqAPF6}) is
equivalent to the following relation
\begin{equation}\label{eqAPF7}
\max\limits_{|s|\le \varepsilon_{0}^{-1},\ s\in \mathbb
T}\rho(\varphi(s+t_m+\tau),\varphi(s+t_m))\ge
\varepsilon_0\nonumber
\end{equation}
for any $m\in \mathbb N$. Since
$[-\varepsilon_{0}^{-1},\varepsilon_{0}^{-1}]\bigcap \mathbb T$ is
a compact subset of $\mathbb T$, then for any $m\in \mathbb N$
there exists a number $s_m\in
[-\varepsilon_{0}^{-1},\varepsilon_{0}^{-1}]\bigcap \mathbb T$
such that
\begin{equation}\label{eqAPF8}
\rho(\varphi(s_m+t_m+\tau),\varphi(s_m+t_m))\ge
\varepsilon_0\nonumber
\end{equation}
for any $m\in \mathbb N$. Denote by $t^{'}_{m}:=s_m+t_m$ ($m\in
\mathbb N$), then without loss of the generality we may assume
that $t^{'}_{m}\to +\infty$ as $m\to +\infty $ because the
sequence $\{s_m\}$ is precompact and $t_m\to +\infty$ as $m\to
\infty$. Thus we have
\begin{equation}\label{eqAPF9}
\rho(\varphi(t^{'}_{m}+\tau),\varphi(t^{'}_{m}))\ge
\varepsilon_{0}
\end{equation}
for any $m\in \mathbb N$. Note that the relations (\ref{eqAPF9})
and (\ref{eqAPF3}) are contradictory. The obtained contradiction
prove our statement.

To prove the implication (\ref{eqAPF4}) $\to$ (\ref{eqAPF3}) we
fix an arbitrary natural number $k$. Note that
\begin{equation}\label{eqAPF10}
\rho(\varphi(t+\tau),\varphi(t))\le \max\limits_{|s|\le k,\ s\in
\mathbb
T}\rho(\varphi(t+s+\tau),\varphi(t+s))=d_{k}(\sigma(t+\tau,\varphi),\sigma(t,\varphi))\to
0\nonumber
\end{equation}
as $t\to +\infty$. Lemma is completely proved.
\end{proof}

\begin{lemma}\label{lAPF4} Let $\varphi \in C(\mathbb T,X)$ be
positively Lagrange stable and $\tau \in \mathbb T$ be a positive
number. The function $\varphi$ is $S$-asymptotically
$\tau$-periodic (respectively, $S$-asymptotically stationary) if
and only if $\omega_{\varphi}\subset C_{\tau}(\mathbb T,X)$, i.e.,
every function $\psi\in \omega_{\varphi}$ is $\tau$-periodic
(respectively, every function $\psi \in \omega_{\varphi}$ is
stationary).
\end{lemma}
\begin{proof}
By Lemma \ref{lAPF3} the function $\varphi \in C(\mathbb T,X)$ is
$S$-asymptotically $\tau$-periodic if and only if the motion
$\sigma(t,\varphi)$ generated by $\varphi$ satisfies relation
(\ref{eqAPF4}). Now to finish the proof it suffices apply Theorem
\ref{th1.3.9}.
\end{proof}

Below we give an example which illustrates Lemma \ref{lAPF4}.

\begin{example}\label{exAPF1} Let $\varphi\in C(\mathbb R_{+},\mathbb R)$ is defined by the equality
$\varphi(t):=\sin \sqrt{\pi^{2}+t}$ for any $t\in \mathbb R_{+}$

\begin{lemma}\label{lAPF5} The function $\psi$ belongs to
$\omega_{\varphi}$ if and only if there exist $\alpha \in [-1,1]$
such that $\psi(t)=\alpha$ for any $t\in \mathbb R$.
\end{lemma}
\begin{proof}
Let $\psi \in \omega_{\varphi}$, then there exists a sequence
$\{t_k\}\subset \mathbb R_{+}$ such that $\varphi^{t_k}\to \psi$
as $k\to \infty$ in the space $C(\mathbb R_{+},\mathbb R)$. Note
that
\begin{equation}\label{eqAPF10.2}
\varphi(t+t_k)=\varphi(t_k)+\int_{0}^{t}a(s+t_k)ds
\end{equation}
for any $t\in \mathbb R_{+}$ and $k\in \mathbb N$, where
$$
a(s):=\varphi'(s)=\frac{1}{2\sqrt{\pi^2+t}}\cos \sqrt{\pi^2+t}
$$
($t\in \mathbb R_{+}$). Since $a(s)\to 0$ as $s\to \infty$, then
$a^{t_k}\to 0$ as $k\to \infty$ in the space $C(\mathbb
R_{+},\mathbb R)$. Passing to the limit in (\ref{eqAPF10.2}) and
taking into account the reasoning above we obtain
$$
\psi(t)=\lim\limits_{k\to \infty}\varphi(t+t_k)=\lim\limits_{k\to
\infty}\varphi(t_k)=\alpha
$$
for any $t\in \mathbb R$.

Conversely. Let $\alpha \in [-1,1]$. Consider the function
$\psi(t):=\alpha$ ($t\in \mathbb R$). We will show that $\psi \in
\omega_{\varphi}$. Indeed, since $\delta_{\varphi}=[-1,1]$ (see
Example \ref{exP1}, item (v)) then there exists a sequence $t_k\to
+\infty$ as $k\to \infty$ such that $\varphi(t_k)\to \alpha$.
Consider the sequence $\varphi^{t_k}$. We will show that
$\varphi^{t_k}\to \psi$ as $k\to \infty$ in the space $C(\mathbb
R_{+},\mathbb R)$, where $\psi(t):=\alpha$ for any $t\in\mathbb
R$. To this end we will use again the relation (\ref{eqAPF10.2}).
Since $\varphi(t_k)\to \alpha$ and $a^{t_k}\to 0$ (in the space
$C(\mathbb R_{+},\mathbb R)$), then passing to the limit in
(\ref{eqAPF10.2}) as $k\to \infty$ we obtain the required
statement. Lemma is proved.
\end{proof}
\end{example}

Let $W$ be a subset of the space $\mathbb R^{n}$. Denote by
$C(\mathbb T\times W, \mathbb R^{n})$ the space of all continuous
functions $f:\mathbb T \times W \to \mathbb R^{n}$ equipped with
the distance
\begin{equation}\label{eqD*1}
d(f,g):=\sup\limits_{L>0}\min\{\max\limits_{|t|+|x|\le L,\
(t,x)\in \mathbb T\times W}\rho(f(t,x),g(t.x)),L^{-1}\}.\nonumber
\end{equation}

Note that $d$ is a complete metric which generates the compact
open topology on the space $C(\mathbb T\times W,\mathbb R^{n})$.
Denote by $(C(\mathbb T\times W,\mathbb R^{n}), \mathbb T,\sigma)$
the shift dynamical system on the space $C(\mathbb T\times
W,\mathbb R^{n})$ (see, for example, \cite[Ch.I]{Che_2015}), i.e.,
$\sigma(h,f):=f^{h}$ and $f^{h}(t,x):=f(t+h,x)$ for any $(t,x)\in
\mathbb T\times W$.

\begin{definition}\label{defD1} A function $f\in C(\mathbb T\times W,\mathbb
R^{n})$ is said to be asymptotically $\tau$-periodic
(respectively, positively Lagrange stable and so on) in $t\in
\mathbb T$ uniformly with respect to $x$ on every compact subset
from $W$ if the motion $\sigma(t,f)$ (defined by function $f$) of
the dynamical system $(C(\mathbb \times W,\mathbb R^{n}),\mathbb
T,\sigma)$ is asymptotically $\tau$-periodic (respectively,
positively Lagrange stable and so on).
\end{definition}


\section{Application}\label{S7}

\subsection{Scalar differential equations}\label{S7.1}

\subsubsection{An analogue of Massera's theorem for scalar
asymptotically $\tau$-periodic differential
equations}\label{S7.1.1}

Consider a scalar differential equation
\begin{equation}\label{eqMT1}
x'=f(t,x),
\end{equation}
where $f\in C(\mathbb T \times \mathbb R,\mathbb R)$ and $\mathbb
T =\mathbb R_{+}$ or $\mathbb R$.

Along with equation (\ref{eqMT1}) consider its $H$-class, i.e.,
the family of equations
\begin{equation}\label{eqMT2}
x'=g(t,x),
\end{equation}
where $g\in H(f):=\overline{\{f^{h}|\ h\in \mathbb T\}}$.

Recall \cite{Sel_1971} that the function $f\in C(\mathbb T \times
\mathbb R,\mathbb R)$ (respectively, equation (\ref{eqMT1})) is
called regular, if for every $(u,g)\in \mathbb R\times H(f)$ the
equation (\ref{eqMT2}) admits a unique solution $\varphi(t,u,g)$
defined on $\mathbb R_{+}$ passing through the point $u$ at the
initial moment, i.e., $\varphi(0,u,g)=u$.

\begin{theorem}\label{thM1} (Massera's theorem for
asymptotically $\tau$-periodic differential equations) Let $f\in
C(\mathbb T\times \mathbb R,\mathbb R)$. Assume that the following
conditions are fulfilled:
\begin{enumerate}
\item the function $f$ is regular; \item the function $f$ is
asymptotically $\tau$-periodic in $t$ uniformly with respect to
$x$ on every compact subset from $\mathbb R$; \item
$\varphi(t,u_0,f)$ is a bounded on $\mathbb R_{+}$ solution of the
equation (\ref{eqMT1}).
\end{enumerate}

Then the solution $\varphi(t,u_0,f)$ is $S$-asymptotically
$\tau$-periodic, i.e.,
\begin{equation}\label{eqM4}
\lim\limits_{t\to
+\infty}|\varphi(t+\tau,u_0,f)-\varphi(t,u_0,f)|=0 .\nonumber
\end{equation}
\end{theorem}
\begin{proof} Denote by $Y:=H(f)$ and $(Y,\mathbb T,\sigma)$ the
shift dynamical system on $H(f)$ induced by $(C(\mathbb T\times
\mathbb R,\mathbb R)$. Let $\langle \mathbb R,\varphi,(Y,\mathbb
T,\sigma)\rangle$ be the cocycle over $(Y,\mathbb T,\sigma)$ with
the fibre $\mathbb R$ generated by the equation (\ref{eqMT1}). The
cocycle $\varphi$ is an one-dimensional and by regularity of $f$
it is strictly monotone. Since the solution $\varphi (t,u_0,f)$ is
bounded and the right hand side $f$ of the equation (\ref{eqMT1})
is positively Lagrange stable, then the point $x_0=(u_0,f)\in
X=\mathbb R\times H(f)$ is positively Lagrange stable of the
skew-product dynamical system $(X,\mathbb R_{+},\pi)$ associated
by the cocycle $\varphi$. Now to finish the proof of Theorem
\ref{thM1} it suffices apply Theorem \ref{thBD4}.
\end{proof}

\subsubsection{Asymptotically $\tau$-periodic
solutions}\label{S7.1.2}

Below we study the problem of existence of asymptotically
$\tau$-periodic solutions of equation (\ref{eqMT1}) if its right
hand side $f$ is asymptotically $\tau$-periodic in time, i.e.,
\begin{enumerate}
\item
\begin{equation}\label{eqPR1}
f(t,x)=P(t,x)+R(t,x),\nonumber
\end{equation}
where $P,R\in C(\mathbb T\times \mathbb R,\mathbb R)$; \item the
function $P$ is $\tau$-periodic in time, that is,
\begin{equation}\label{eqPR1.1}
P(t+\tau,x)=P(t,x)\nonumber
\end{equation}
for any $(t,x)\in \mathbb T\times \mathbb R$ and
\item
\begin{equation}\label{eqPR1.2}
 \lim\limits_{t\to +\infty}|R(t,x)|=0\nonumber
 \end{equation}
uniformly with respect to $x$ on every compact subset from
$\mathbb R$.
\end{enumerate}

If the function $f$ is regular, then by Theorem \ref{thM1} every
bounded on $\mathbb R_{+}$ solution $\varphi(t,u_0,f)$ ($u_0\in
\mathbb R$) of equation (\ref{eqMT1}) is $S$-asymptotically
$\tau$-periodic. On the other hand there exist Examples of type
(\ref{eqMT1}) with asymptotically $\tau$-periodic right hand side
which has no asymptotically $\tau$-periodic solution (see Example
\ref{exP1}). Thus, in order for equation (\ref{eqMT1}) with an
asymptotically $\tau$-periodic right-hand side to have at least
one asymptotically $\tau$-periodic solution, it is necessary to
impose some additional conditions.

\begin{definition}\label{defSI1} A $\tau$-periodic solution
$\varphi(t,u_0,p)$ of the $\tau$-periodic equation
\begin{equation}\label{eqMPT1}
x'=P(t,x)
\end{equation}
is said to be
\begin{enumerate}
\item isolated if there exists a positive number $\delta$ such
that the segment $(u_0-\delta,u_0+\delta)$ does not contain points
$u$ other than $u_0$ such that the solution $\varphi(t,u,q)$ of
(\ref{eqMPT1}) is $\tau$-periodic; \item positively stable if for
arbitrary positive number $\varepsilon$ there exists a number
$\delta =\delta(\varepsilon)>0$ such that $|u-u_0|<\delta$ implies
$|\varphi(t,u,p)-\varphi(t,u_0,p)|<\varepsilon$ for any $t\in
\mathbb R_{+}$; \item positively attracting if there exists a
positive number $\gamma$ such that $|u-u_0|<\gamma$ implies
$\lim\limits_{t\to +\infty}|\varphi(t,u,p)-\varphi(t,u_0,p)|=0$;
\item positively asymptotically stable if its is positively stable
and positively attracting.
\end{enumerate}
\end{definition}


\begin{theorem}\label{thASL1} Let $f\in C(\mathbb T\times \mathbb R,\mathbb R)$ be a regular function. Assume that the following conditions
are fulfilled:
\begin{enumerate}
\item $f$ is asymptotically $\tau$-periodic in time; \item the
periodic solutions of equation (\ref{eqMPT1}) are isolated; \item
$\varphi(t,u_0,f)$ is a bounded on $\mathbb R_{+}$ solution of
equation (\ref{eqMT1}).
\end{enumerate}

Then the solution $\varphi(t,u_0,f)$ of (\ref{eqMT1}) is
asymptotically $\tau$-periodic.
\end{theorem}
\begin{proof}
Let $\langle \mathbb R,\varphi,(H(f),\mathbb R_{+},\sigma)\rangle$
be the cocycle over $(H(f),\mathbb R_{+},\sigma)$ with the fibre
$\mathbb R$ generated by the equation (\ref{eqMT1}). The cocycle
$\varphi$ is one-dimensional and by regularity of $f$ it is
strictly monotone. Since the solution $\varphi (t,u_0,f)$ is
bounded and the right hand side $f$ of the equation (\ref{eqMT1})
is positively Lagrange stable, then the point $x_0=(u_0,f)\in
X=\mathbb R\times H(f)$ is positively Lagrange stable of the
skew-product dynamical system $(X,\mathbb R_{+},\pi)$ associated
by the cocycle $\varphi$.

Since the periodic solutions of the equation (\ref{eqMPT1}) are
isolated, then the periodic points $(u,q)\in X=\mathbb R\times
H(f)$ of the skew-product dynamical system $(X,\mathbb R_{+},\pi)$
are isolated. Now to finish the proof of Theorem \ref{thASL1} it
suffices apply Theorem \ref{thBD5}.
\end{proof}

\begin{coro}\label{corASL1} Let $f\in C(\mathbb T\times \mathbb R,\mathbb R)$ be a regular function. Assume that the following conditions
are fulfilled:
\begin{enumerate}
\item $f$ is asymptotically $\tau$-periodic in time; \item every
periodic solution of equation (\ref{eqMPT1}) is asymptotically
stable (positively or negatively); \item $\varphi(t,u_0,f)$ is a
bounded on $\mathbb R_{+}$ solution of the equation (\ref{eqMT1}).
\end{enumerate}

Then the solution $\varphi(t,u_0,f)$ of (\ref{eqMT1}) is
asymptotically $\tau$-periodic.
\end{coro}
\begin{proof} Since every $\tau$-periodic solution of the equation
(\ref{eqMPT1}) is asymptotically stable (positively or
negatively), then this statement follows directly from Theorem
\ref{thASL1} (see also Corollary \ref{corBD6}).
\end{proof}

\begin{remark}\label{remASL1} This statement refines the main
result from \cite{HKFX_2006}, where it was proved (Theorem 3.2)
that if the equation (\ref{eqMPT1}) is convergent (it has a unique
$\tau$-periodic globally asymptotically stable solution), then
every bounded on $\mathbb R_{+}$ solution of the equation
(\ref{eqMT1}) is asymptotically $\tau$-periodic.
\end{remark}

\begin{coro}\label{corASL2} Let $f\in C(\mathbb T\times \mathbb R,\mathbb R)$ be a regular function. Assume that the following conditions
are fulfilled:
\begin{enumerate}
\item $f$ is asymptotically $\tau$-periodic in time, i.e., there
exist functions $P,R\in C(\mathbb T\times \mathbb R,\mathbb R)$
such that
\begin{enumerate}
\item $f(t,x)=P(t,x)+R(t,x)$ for any $(t,x)\in \mathbb T\times
\mathbb R$; \item $P(t+\tau,x)=P(t,x)$ for any $(t,x)\in \mathbb
T\times \mathbb R$; \item $\lim\limits_{t\to +\infty}|R(t,x)|=0$
uniformly with respect to $x$ on every compact subset from
$\mathbb R$; \item $P$ is a polynomial in $x$, i.e.,
$$
P(t,x)=x^n+P_1(t)x^{n-1}+\ldots +P_{n-1}(t)x+P_n(t)
$$
for any $(t,x)\in \mathbb T\times \mathbb R$ and the functions
$P_{i}\in C_{\tau}(\mathbb T,\mathbb R)$ ($i=1,\ldots,n$).
\end{enumerate}
\item $\varphi(t,u_0,f)$ is a bounded on $\mathbb R_{+}$ solution
of the equation (\ref{eqMT1}).
\end{enumerate}

Then the solution $\varphi(t,u_0,f)$ of (\ref{eqMT1}) is
asymptotically $\tau$-periodic.
\end{coro}
\begin{proof} By Theorem 9.5 \cite[Ch.II]{Plis_64} the equation
\begin{equation}\label{eqMPPT1}
x'=P(t,x)=x^n+P_1(t)x^{n-1}+\ldots +P_{n-1}(t)x+P_n(t)\nonumber
\end{equation}
has at most a finite number of $\tau$-periodic solutions and,
consequently, they are isolated. Now to finish the proof of this
statement it suffices apply Theorem \ref{thASL1}.
\end{proof}

\subsection{Scalar Difference Equations}\label{S7.2}

This subsection is dedicated to the study the problem of existence
of asymptotically periodic solutions for scalar difference
equation of the form
\begin{equation}\label{eqDF1}
x(t+1)=f(t,x(t)) ,
\end{equation}
where $f\in C(\mathbb T \times \mathbb R,\mathbb R)$, $\mathbb
T=\mathbb Z_{+}$ or $\mathbb Z$.

Along with equation (\ref{eqDF1}) consider its $H$-class, i.e.,
the family of equations
\begin{equation}\label{eqDF2}
x(t+1)=g(t,x(t)) ,
\end{equation}
where $g\in H(f):=\overline{\{g^{h}|\ h\in\mathbb T\}}$ and
$g^{h}(t,x):=g(t+h,x)$ for any $(t,x)\in \mathbb T\times \mathbb
R$.

\begin{example}\label{ex2.1} {\em Consider the equation (\ref{eqDF1}),
where $f\in C(\mathbb T\times \mathbb R,\mathbb R)$.

Denote by $Y:=H(f)$, $(H(f),\mathbb T,\sigma)$ the shift dynamical
system on $H(f)$ and $\varphi(t,v,g)$ the solution of equation
(\ref{eqDF2}) with initial condition $\varphi(0,v,g)=v.$ From the
general properties of difference equations it follows that:
\begin{enumerate}
\item $\varphi(0,v,g)=u$ for any $u\in \mathbb R$ and $g\in H(f);$
\item $\varphi (t+\tau,v,g)=\varphi (t,\varphi
(\tau,v,g),\sigma(\tau,g))$ for all $t,\tau \in \mathbb Z_{+}$ and
$(u,y)\in \mathbb Z_{+}\times Y$; \item the mapping $\varphi$ is
continuous.
\end{enumerate}

Thus every equation (\ref{eqDF1}) generates a cocycle $\langle
\mathbb R,\varphi, (Y,\mathbb T,\sigma)\rangle$ over $(Y,\mathbb
T,\sigma)$ with fibre $\mathbb R$.}
\end{example}

\begin{lemma}\label{l8.2}\cite{CM_2005} Let $f_i:\mathbb Z_{+}\times \mathbb R
\to \mathbb R\ (i=1,2).$
Suppose that the following conditions hold:
\begin{enumerate}
\item
$u_1,u_2\in \mathbb R$ and $u_1\le u_2$ (respectively, $u_1 <u_2$);
\item
$f_1(t,x)\le f_2(t,x)$ (respectively, $f_1(t,x)<f_2(t,x)$) for all $t\in \mathbb Z_{+}$
and $x\in \mathbb R;$
\item
the function $f_2$ is monotone non-decreasing (respectively, strictly monotone increasing)
with respect to variable $u\in \mathbb R.$
\end{enumerate}

Then $\varphi(t,u_1,f_1) \le \varphi(t,u_2,f_2)$ (respectively,
$\varphi(t,u_1,f_1) < \varphi(t,u_2,f_2)$)
for all $t\in \mathbb Z_{+}$.
\end{lemma}

\begin{example}\label{exBH1}
Consider the Beverton-Holt equation
\begin{equation}\label{eq8.1}
u(t+1)=\frac{\mu K(t) u(t)}{K(t) +(\mu -1)u(t)},
\end{equation}
where $K(t)$ is an asymptotically $\tau$-periodic sequence. Assume
that $\mu >1$ and there exist positive numbers $\alpha, \beta$
such hat $\alpha \le K(t)\le \beta$ for any $t\in \mathbb Z_{+}$.

Let
\begin{equation}\label{eqBH1}
f(t,u):=\frac{\mu K(t) u}{K(t) +(\mu -1)u},\nonumber
\end{equation}
then for any $g\in H(f)$ there exists a function $G\in H(K)$ such
that
\begin{equation}\label{eqBH1.1}
 g(t,u):=\frac{\mu G(t) u}{G(t) +(\mu -1)u},
\end{equation}
then for any $(t,u)\in \mathbb T\times \mathbb R_{+}$.

Since
$$
g'_{u}(t,u)=\frac{\mu G(t)^{2} }{(G(t) +(\mu -1)u)^{2}}\ge
\frac{\mu \alpha^{2}}{(\beta +(\mu-1)u)^2}>0
$$
for any $t\in \mathbb Z_{+}$ and $u\in \mathbb R_{+}$, then the
function $g$ defined by the equality (\ref{eqBH1.1}) is strictly
monotone increasing with respect to $u\in\mathbb R_{+}$.
\end{example}

\subsubsection{Massera-type theorem for scalar asymptotically
$\tau$-periodic difference equations}\label{S7.2.1}

\begin{theorem}\label{thDM1} (Massera's theorem for asymptotically
$\tau$-periodic difference equations) Let $f\in
C(\mathbb T\times \mathbb R,\mathbb R)$. Assume that the following
conditions are fulfilled:
\begin{enumerate}
\item the function $f$ is asymptotically $\tau$-periodic in $t$
uniformly with respect to $x$ on every compact subset from
$\mathbb R$; \item every function $g\in H^{+}(f)$ is strictly
monotone, i.e., $v_1<v_2$ ($v_1,v_2\in \mathbb R$) implies
$g(t,v_1)<g(t,v_2)$ for any $(t,g)\in \mathbb T\times H^{+}(f)$;
\item $\varphi(t,u_0,f)$ is a bounded on $\mathbb Z_{+}$ solution
of equation (\ref{eqDF1}).
\end{enumerate}

Then the solution $\varphi(t,u_0,f)$ is $S$-asymptotically
$\tau$-periodic, i.e.,
\begin{equation}\label{eqDM4}
\lim\limits_{t\to
+\infty}|\varphi(t+\tau,u_0,f)-\varphi(t,u_0,f)|=0 .\nonumber
\end{equation}
\end{theorem}
\begin{proof} Denote by $Y:=H^{+}(f)$ and $(Y,\mathbb T,\sigma)$ the
shift dynamical system on $H^{+}(f)$ induced by $(C(\mathbb
T\times \mathbb R,\mathbb R)$. Let $\langle \mathbb
R,\varphi,(Y,\mathbb T,\sigma)\rangle$ be the cocycle over
$(Y,\mathbb T,\sigma)$ with the fibre $\mathbb R$ generated by the
equation (\ref{eqDF1}). The cocycle $\varphi$ is an
one-dimensional and by (\ref{eqBH1.1}) it is strictly monotone.
Since the solution $\varphi (t,u_0,f)$ is bounded and the right
hand side $f$ of the equation (\ref{eqDF1}) is positively Lagrange
stable, then the point $x_0=(u_0,f)\in X=\mathbb R\times H^{+}(f)$
is positively Lagrange stable of the skew-product dynamical system
$(X,\mathbb R_{+},\pi)$ associated by the cocycle $\varphi$. Now
to finish the proof of Theorem \ref{thDM1} it suffices apply
Theorem \ref{thBD4}.
\end{proof}

Below we give an example which show that under the conditions of
Theorem \ref{thDM1} the bounded on $\mathbb Z_{+}$ solution
$\varphi(t,u_0,f)$, generally speaking, is not asymptotically
$\tau$-periodic.

\begin{example}\label{exDP1}
Let $a\in C(\mathbb R_{+},\mathbb R)$ defined by the equality
\begin{equation}\label{eqP11}
a(t)=\frac{1}{2\sqrt{\pi^2+t}}\cos{\sqrt{\pi^2 +t}}\nonumber
\end{equation}
for any $t\in \mathbb R_{+}$. Denote by $A\in C(\mathbb
Z_{+},\mathbb R)$ the sequence defined by
$$
A(k):=\int_{0}^{1}a(k+s)ds=\sin\sqrt{\pi^2+t}\Big{|}_{k}^{k+1}=
$$
$$
2\sin\frac{1}{2(\sqrt{\pi^2+k+1}+\sqrt{\pi^2+k})}\sin\frac{\sqrt{\pi^2+k+1}
+\sqrt{\pi^2+k}}{2}
$$
for any $k\in \mathbb Z_{+}$.

\begin{lemma}\label{lDP1}
The following statements hold:
\begin{enumerate}
\item
\begin{equation}\label{eqDP12}
\lim\limits_{k\to +\infty}|A(k)|=0;\nonumber
\end{equation}
\item $|\varphi(k)|\le 1$ for any $k\in \mathbb Z_{+}$, where
\begin{equation}\label{eqP13}
\varphi(k)=\int_{0}^{k}a(s)ds=\sin{\sqrt{\pi^2 +k}} ;\nonumber
\end{equation}
\item $\tilde{\delta}_{\varphi}\subseteq [-1,1],$ where
$\tilde{\delta}_{\varphi}:=\{x\in \mathbb R|$ there exists a
sequence $\{k_n\}\subset \mathbb N$ such that $k_n\to \infty$ as
$n\to \infty$ and $x=\lim\limits_{n\to \infty}\varphi(k_n)\}$;
 \item the set $\{\varphi^{h}|\ h\in \mathbb Z_{+}\}$ is precompact
in $C(\mathbb Z_{+},\mathbb R)$; \item for any $\psi \in
\omega_{\varphi}$ there exists a constant $\alpha \in
\tilde{\delta}_{\varphi}$ such that $\psi(k)=\alpha$ for any $k\in
\mathbb Z$; \item the set $\tilde{\delta}_{\varphi}$ contains at
least two different numbers from $[-1,1]$.
\end{enumerate}
\end{lemma}
\begin{proof} The first four statement are evident.

Let $\psi\in \omega_{\varphi}$, then there exists a sequence
$\{h_n\}\subset \mathbb Z_{+}$ such that $h_n\to +\infty$ and
$\varphi^{h_n}\to \psi$ in $C(\mathbb Z_{+},\mathbb R)$ as $n\to
\infty$. Note that
\begin{equation}\label{eqP13.1}
\varphi(k+h_n)=\varphi(h_n)+\int_{0}^{k}a^{h_n}(s)ds
\end{equation}
for any $(k,n)\in \mathbb Z_{+}\times \mathbb N$. Passing to the
limit in the equality (\ref{eqP13.1}) as $n\to \infty$ and taking
into account that $a^{h_n}\to 0$ in the space $C(\mathbb
R_{+},\mathbb R)$ we obtain $\psi(k)=\psi(0):=\alpha \in
\tilde{\delta}_{\varphi}$ for any $k\in \mathbb Z$.

To prove the sixth statement we consider the difference equation
\begin{equation}\label{eqDP14}
\Delta x(k)=A(k),
\end{equation}
where $\Delta x(k):=x(k+1)-x(t)$ for any $k\in\mathbb Z_{+}$.
Along with the equation (\ref{eqDP14}) consider its $H^{+}$-class,
i.e., the family of equations
\begin{equation}\label{eqDP15}
\Delta y(k)=B(k),
\end{equation}
where $B\in H^{+}(A):=\overline{\{A^{h}|\ h\in \mathbb Z_{+}\}}$,
$A^{h}$ is the $h$-translation of $A\in C(\mathbb Z_{+},\mathbb
R)$ and by bar the closure in the space $C(\mathbb Z_{+},\mathbb
R)$ is denoted. Since $A(k)\to 0$ as $k\to +\infty$, then the
$\omega$-limit set $\omega_{A}$ consists of a unique stationary
point (stationary sequence) $\theta$. Thus for the equation
(\ref{eqDP14}) we have a unique $\omega$-limit equation
\begin{equation}\label{eqDP16}
\Delta y(k)=0 .\nonumber
\end{equation}
Denote by $\varphi(k,v,b)$ the unique solution of equation
(\ref{eqDP15}) passing through the point $v\in \mathbb R$ at the
initial moment $k=0$ and defined on $\mathbb Z_{+}$. It is clear
that
$$
\varphi(k,v,b)=v+\sum_{0}^{k-1}B(s)=v+\int_{0}^{k}b(s)ds
$$
for any $k\in \mathbb Z_{+}$ ($B(k)=\int_{0}^{1}b(k+s)ds$, where
$b\in H^{+}(a)$). Let $Y:=H^{+}(A)$ and $(Y,\mathbb Z_{+},\sigma)$
be the shift dynamical system on $H^{+}(A)$. It easy to check that
the triplet $\langle \mathbb R, \varphi, (Y,\mathbb
Z_{+},\sigma)\rangle$ is a cocycle over $(Y,\mathbb Z_{+},\sigma)$
with the fibre $\mathbb R$.

Consider the skew-product dynamical system $(X,\mathbb Z_{+},\pi)$
($X:=\mathbb R_{+}\times H^{+}(A)$ and $\pi :=(\varphi,\sigma)$)
generated by cocycle $\varphi$. Let $x_0:=(0,A)\in X$, then
$\pi(k,x_0)=(\varphi(k),A^{k})$ for any $k\in \mathbb Z_{+}$. It
is clear that $\{\pi(k,x_0)|\ k\in \mathbb Z_{+}\}$ is precompact
and by Theorem \ref{thBD4} the set $M=\omega_{x_0}\bigcap X_{q}$
($q=\theta$, $\omega_{A}=\{\theta\}$) consists of a stationary
points. We will show that the set $M$ contains more than one
point. Indeed, if we suppose that it is not true, then there
exists $(\alpha_{0},\theta)\in M$ such that $\alpha_{0}\in
[-1.1]$. This means that $\pi(k,x_0)\to (\alpha_0,\theta)$ as
$k\to \infty$ and, consequently,
\begin{equation}\label{eqDP17}
\lim\limits_{k\to \infty}\varphi(k)=\alpha_0,
\end{equation}
where $\varphi(k)=\varphi(k,0,a)$ for any $k\in\mathbb
Z_{+}$.\emph{} We will show that from (\ref{eqDP17}) follows
\begin{equation}\label{eqDP17.1}
\lim\limits_{t\to \infty}\varphi(t)=\alpha_0 .\nonumber
\end{equation}
Indeed, let $\{t_n\}\subset \mathbb R_{+}$ be an arbitrary
sequence such that $t_n\to +\infty$ as $n\to \infty$. We have
$t_n=k_n+\tau_{n}$ for any $n\in \mathbb N$, $k_n\in \mathbb
Z_{+}$ and $\tau_{n}\in [0,1)$. Without loss of generality we can
suppose that the sequence $\{\tau_n\}\subseteq [0,1]$ converges.
Denote its limit by
\begin{equation}\label{eqDP20}
\tau_0=\lim\limits_{m\to \infty}\tau_{m},\nonumber
\end{equation}
then we have
\begin{equation}\label{eqDP20.1}
\varphi(t_n)=\varphi(k_n+\tau_n)=\varphi(k_n)+\int_{0}^{\tau_n}a^{k_n}(s)ds
\end{equation}
for any $n\in\mathbb Z_{+}$. Passing to the limit in
(\ref{eqDP20.1}) and taking into account that $a^{h}\to \theta$ as
$h\to \infty$ in the space $C(\mathbb R_{+},\mathbb R)$ and
(\ref{eqDP17}) we obtain
\begin{equation}\label{eqDP20.2}
\lim\limits_{n\to \infty}\varphi(t_n)=\alpha_{0},
\end{equation}
i.e., $\delta_{\varphi}=\{\alpha_{0}\}$. On the other hand by
Example \ref{exP1} (item (v)) $\delta_{\varphi}=[-1,1]$. The
obtained contradiction proves the sixth statement. Lemma is
completely proved.
\end{proof}
The required example is constructed.
\end{example}

\subsubsection{Asymptotically $\tau$-periodic solutions of
difference equations}\label{S7.2.2}

Below we study asymptotically $\tau$-periodic solutions of the
equation (\ref{eqDF1}) if its right hand side $f$ is
asymptotically $\tau$-periodic in time, i.e.,
\begin{equation}\label{eqDPR1}
f(t,x)=P(t,x)+R(t,x),\nonumber
\end{equation}
where $P,R\in C(\mathbb Z\times \mathbb R,\mathbb R)$,
$P(t+\tau,x)=P(t,x)$ for any $(t,x)\in \mathbb Z\times \mathbb R$
and $\lim\limits_{t\to +\infty}|R(t,x)|=0$ uniformly with respect
to $x$ on every compact subset from $\mathbb R$.

By Theorem \ref{thDM1} every bounded on $\mathbb Z_{+}$ solution
$\varphi(t,u_0,f)$ ($u_0\in \mathbb R$) of equation (\ref{eqDF1})
is $S$-asymptotically $\tau$-periodic. On the other hand there
exist Examples of type (\ref{eqDF1}) with asymptotically
$\tau$-periodic right hand side which has no asymptotically
$\tau$-periodic solution (see Example \ref{exDP1}). Thus, in order
for the equation (\ref{eqDF1}) with an asymptotically
$\tau$-periodic right-hand side to have at least one
asymptotically $\tau$-periodic solution, it is necessary to impose
some additional conditions.

\begin{definition}\label{defDSI1} A $\tau$-periodic solution
$\varphi(t,u_0,p)$ of the $\tau$-periodic equation
\begin{equation}\label{eqDMPT1}
x(t+1)=P(t,x(t))
\end{equation}
is said to be
\begin{enumerate}
\item isolated if there exists a positive number $\delta$ such
that the segment $(u_0-\delta,u_0+\delta)$ does not contain points
$u$ other than $u_0$ such that the solution $\varphi(t,u,q)$ of
(\ref{eqDMPT1}) is $\tau$-periodic; \item positively stable if for
arbitrary positive number $\varepsilon$ there exists a number
$\delta =\delta(\varepsilon)>0$ such that $|u-u_0|<\delta$ implies
$|\varphi(t,u,p)-\varphi(t,u_0,p)|<\varepsilon$ for any $t\in
\mathbb Z_{+}$; \item positively attracting if there exists a
positive number $\gamma$ such that $|u-u_0|<\gamma$ implies
$\lim\limits_{t\to +\infty}|\varphi(t,u,p)-\varphi(t,u_0,p)|=0$;
\item positively asymptotically stable if its is positively stable
and positively attracting.
\end{enumerate}
\end{definition}

\begin{theorem}\label{thDASL1} Let $f\in C(\mathbb Z\times \mathbb R,\mathbb R)$. Assume that the following conditions
are fulfilled:
\begin{enumerate}
\item $f$ is asymptotically $\tau$-periodic in time; \item the
periodic solutions of equation (\ref{eqDMPT1}) are isolated; \item
every function $g\in H^{+}(f)$ is strictly monotone; \item
$\varphi(t,u_0,f)$ is a bounded on $\mathbb Z_{+}$ solution of
equation (\ref{eqDF1}).
\end{enumerate}

Then the solution $\varphi(t,u_0,f)$ of (\ref{eqDF1}) is
asymptotically $\tau$-periodic.
\end{theorem}
\begin{proof}
Let $\langle \mathbb R,\varphi,(H^{+}(f),\mathbb
Z_{+},\sigma)\rangle$ be the cocycle over $(H^{+}(f),\mathbb
Z_{+},\sigma)$ with the fibre $\mathbb R$ generated by the
equation (\ref{eqDF1}). The cocycle $\varphi$ is one-dimensional
and by condition (iii) it is strictly monotone. Since the solution
$\varphi (t,u_0,f)$ is bounded and the right hand side $f$ of
equation (\ref{eqDF1}) is positively Lagrange stable, then the
point $x_0=(u_0,f)\in X=\mathbb R\times H(f)$ is positively
Lagrange stable of the skew-product dynamical system $(X,\mathbb
Z_{+},\pi)$ associated by the cocycle $\varphi$.

Since the periodic solutions of equation (\ref{eqDMPT1}) are
isolated, then the periodic points $(u,q)\in X=\mathbb R\times
H(f)$ of the skew-product dynamical system $(X,\mathbb R_{+},\pi)$
are isolated. Now to finish the proof of Theorem \ref{thDASL1} it
suffices apply Theorem \ref{thBD5}.
\end{proof}

\begin{coro}\label{corDASL1} Let $f\in C(\mathbb Z_{+}\times \mathbb R,\mathbb R)$. Assume that the following conditions
are fulfilled:
\begin{enumerate}
\item $f$ is asymptotically $\tau$-periodic in time; \item every
$\tau$-periodic solution of equation (\ref{eqDMPT1}) is
asymptotically stable; \item $\varphi(t,u_0,f)$ is a bounded on
$\mathbb Z_{+}$ solution of equation (\ref{eqDF1}).
\end{enumerate}

Then the solution $\varphi(t,u_0,f)$ of (\ref{eqDF1}) is
asymptotically $\tau$-periodic.
\end{coro}
\begin{proof} Since the $\tau$-periodic solutions of equation
(\ref{eqDMPT1}) are asymptotically stable (positively or
negatively), then by Theorem \ref{thBD6} the solution
$\varphi(t,u_0,f)$ of (\ref{eqDF1}) is asymptotically
$\tau$-periodic.
\end{proof}

Finally, below we give an example illustrating Theorem
\ref{thDASL1} and Corollary \ref{corDASL1}.

\begin{example}\label{exBH2}
Consider the Beverton-Holt equation (\ref{eq8.1}), where $K(t)$ is
an asymptotically $\tau$-periodic sequence.

\begin{lemma}\label{lBH2} Assume that the following conditions are
fulfilled:
\begin{enumerate}
\item $\mu >1$; \item there exist positive numbers $\alpha$ and
$\beta$ such hat $\alpha \le K(t)\le \beta$ for any $t\in \mathbb
Z_{+}$.
\end{enumerate}

Then every solution $\varphi(t,u,K)$ of equation (\ref{eq8.1}) is
asymptotically $\tau$-periodic.
\end{lemma}
\begin{proof} Denote by $\varphi(t,u,K)$ the unique solution of
equation (\ref{eq8.1}) passing through the point $u\in \mathbb
R_{+}$. It is known (see for example \cite{CM_2005}) that
$\varphi(t,u,K)\ge 0$ for any $t\in \mathbb Z_{+}$. Additionally
by Corollary 4.4. \cite{CM_2005} we have
$$
\lim\sup\limits_{t\to \infty}\varphi(t,u,K)\le \frac{\mu}{\mu
-1}\beta .
$$

Since the sequence $K(t)$ ($t\in \mathbb Z_{+}$) is asymptotically
$\tau$ ($\tau\in \mathbb N$) periodic, then there exist two
sequences $P(t)$ and $R(t)$ such that
\begin{enumerate}
\item[a.] $P(t+\tau)=P(t)$ for any $t\in \mathbb Z_{+}$; \item[b.]
$\lim\limits_{t\to \infty}|R(t)|=0$; \item[c.] $K(t)=P(t)+R(t)$
for any $t\in \mathbb Z_{+}$.
\end{enumerate}

Consider the limiting equation
\begin{equation}\label{eqBHEQ2}
x(t+1)=\frac{\mu P(t) x(t)}{P(t) +(\mu -1)x(t)}\nonumber
\end{equation}
for the equation (\ref{eq8.1}). It is well known (see, for
example, \cite{ES_2006}) under the conditions of Lemma \ref{lBH2}
has two $\tau$-periodic nonnegative solutions: one of them is
trivial (and negatively asymptotically stable) other one is
strongly positive (and positively asymptotically stable). Now to
finish the proof of Lemma it suffices apply Corollary
\ref{corDASL1}. Lemma is proved.
\end{proof}
\end{example}

\end{document}